\documentclass[11pt]{article}

\usepackage[a4paper,vmargin={3.5cm,3.5cm},hmargin={2.5cm,2.5cm}]{geometry}
\usepackage[font=sf, labelfont={sf,bf}, margin=1cm]{caption}
\usepackage{dsfont,amsmath,amsfonts,amsthm,amssymb}

\usepackage[toc,page]{appendix} 
\usepackage{graphicx}
\usepackage{graphics}
\usepackage{stmaryrd}
\usepackage{verbatim}
\usepackage{color}
\usepackage{bbm}
\usepackage[T1]{fontenc}
\usepackage[cyr]{aeguill}

\title{Local convergence of large critical multi-type Galton-Watson trees and applications to random maps}
\author{\text{Robin Stephenson}\thanks{Institut f\"ur Mathekmatik, Universit\"at Z\"urich, Winterthurerstrasse 190, CH-8057 Z\"urich, Switzerland. \newline E-mail: robin.stephenson@normalesup.org}   }
\begin{document}

\maketitle

\newtheorem{theo}{Theorem}[section]
\newtheorem{lemma}{Lemma}[section]
\newtheorem{prop}{Proposition}[section]
\newtheorem{cor}{Corollary}[section]
\newtheorem{defi}{Definition}[section]
\newtheorem{rem}{Remark}
\numberwithin{equation}{section}
\newcommand{\e}{{\mathrm e}}
\newcommand{\rep}{{\mathrm{rep}}}
\newcommand{\R}{{\mathbb{R}}}
\newcommand{\C}{{\mathbb{C}}}
\newcommand{\T}{\mathbf{T}}
\newcommand{\E}{\mathbf{E}}
\renewcommand{\L}{\mathbf{L}}
\newcommand{\TT}{\mathbb{T}}
\newcommand{\s}{\mathcal{S}^{\downarrow}}
\newcommand{\p}{\mathcal{P}}
\newcommand{\N}{\mathbb{N}}
\newcommand{\D}{\mathcal{D}}
\newcommand{\W}{\mathcal{W}}
\newcommand{\pr}{\mathbb{P}}
\newcommand{\z}{\mathbf{\zeta}}
\newcommand{\m}{\mathbf{\mu}}
\newcommand{\M}{\mathcal{M}}
\newcommand{\Z}{\mathbb{Z}}
\newcommand{\dia}{\diamond}
\newcommand{\bul}{\bullet}
\renewcommand{\geq}{\geqslant}
\renewcommand{\leq}{\leqslant}
\renewcommand{\epsilon}{\varepsilon}

\begin{quote}{\small We show that large critical multi-type Galton-Watson trees, when conditioned to be large, converge locally in distribution to an infinite tree which is analogous to Kesten's infinite monotype Galton-Watson tree. This is proven when we condition on the number of vertices of one fixed type, and with an extra technical assumption if we count at least two types. We then apply these results to study local limits of random planar maps, showing that large critical Boltzmann-distributed random maps converge in distribution to an infinite map.}\end{quote}

\section{Introduction}
A planar map is a proper embedding of a finite connected planar graph in the sphere, taken up to orientation-preserving homeomorphisms. These objects were first studied from a combinatorial point of view in the works of Tutte in the 1960s (see notably \cite{Tutteplanarmaps}), and have since been of use in different domains of mathematics, such as algebraic geometry (see for example \cite{LZ}) and theoretical physics (as in \cite{ADJ}). There has been great progress in their probabilistic study ever since the work of Schaeffer \cite{Schaeffer}, which has amongst other things led to finding the scaling limit of many large random maps (we mention \cite{M13} and \cite{LG13}).

Our subject of interest here is the local convergence of large random maps, which means that we are not interested in scaling limits but in the combinatorial structure of a map around a chosen root. Such problems were first studied by Angel and Schramm (\cite{AS}) and Krikun (\cite{Krikun}), who showed that the distributions of uniform triangulations and quadrangulations with $n$ vertices converge weakly as $n$ goes to infinity. Each limit is the distribution of an infinite random map, respectively known as  the uniform infinite planar triangulation (UIPT) and the uniform infinite planar quadrangulation (UIPQ). Of particular interest to us is the paper \cite{CMM} where the convergence to the UIPQ is shown by a method involving the well-known Cori-Vauquelin-Schaeffer bijection (\cite{Schaeffer}).

We will generalize this to a large family of random maps called the class of Boltzmann-distributed random maps. Let $\mathbf{q}=(q_n)_{n\in\N}$ be a sequence of nonnegative numbers. We assign to every finite planar map a weight which is equal to the product of the weights of its faces, the weight of a face being $q_d$ where $d$ is the number of edges adjacent to said face, counted with multiplicity. If the sum of all the weights of all the maps is finite, then one can normalize this into a probability distribution. 

The use of the so-called Bouttier-Di Francesco-Guitter bijection (see \cite{BDFG}, or Section \ref{sec:BDFG}) allows us to obtain the convergence to infinite maps for a fairly large class of weight sequences $\mathbf{q}$. For $\mathbf{q}$ in this class, let $(M_n,E_n)$ be a $\mathbf{q}$-Boltzmann rooted map conditioned to have $n$ vertices (or $n$ edges or $n$ faces) our main Theorem \ref{chap4:thcvcartes} states that this sequence converges in distribution to a random map $(M_{\infty},E_{\infty})$, which we call the \emph{infinite $\mathbf{q}$-Boltzmann map}. Due to combinatorial reasons, we have to restrict $n$ to a sublattice of $\Z_+$.

The classes of weight sequences for which this is true are the class of \emph{critical} sequences when we condition by the number of vertices, and \emph{regular critical} when we condition by the number of edges or faces (both are defined in Section \ref{chap4:defboltzmann}). These classes contain all sequences with finite support (up to multiplicative constants). Taking $q_n=\mathbbm{1}_{\{n=p\}}$ with $p\geq 3$ gives us the case of the uniform $p$-angulation, making our results an extension of what was known about the UIPT and UIPQ.

Local limits of Boltzmann random maps have notably been studied recently in \cite{BS}. A key difference with our work here is the fact that the maps are supposed to be bipartite in \cite{BS} (the weight sequence $\mathbf{q}$ is supported on the even integers). In this context, conditioning maps by their number of edges ends up being more natural than in our work, and it is sufficient to only assume criticality and not regular criticality.

\medskip

The proof of convergence to an infinite map hinges on a similar result for critical multi-type Galton-Watson trees and forests, Theorem \ref{cvforests}. This theorem itself generalizes the well-known fact that critical monotype Galton-Watson trees, when conditioned to be large, converge to an infinite tree formed by a unique infinite spine to which many finite trees are grafted. This infinite tree was first indirectly mentioned in \cite{Kesten1986}, Lemma 1.14, and many details about the convergence are given in \cite{AD14} and \cite{Janson}. One of its properties is that one can obtain its distribution from the distribution of the finite tree by a size-biasing process, as explained in \cite{LyonsPeres}.

In the multi-type case, as with maps, we have two different kinds of conditionings. If the tree is simply critical, then we must condition it by the number of vertices of one type early, while if it is regular critical (criticality and regular criticality being defined in Section \ref{sec:crittree}), we can condition it by its ``size", for a general notion of size where we count all the vertices, giving some integer weight to each vertex depending on its type. The distribution of the infinite limiting tree can once again be described by a biasing process from the original tree, as explained in Proposition \ref{definftree}, something which was anticipated in \cite{KLPP}.


A fairly important issue in Theorem \ref{cvforests} is the problem of periodicity: as with maps, a multi-type Galton-Watson tree cannot have any number of vertices. To be precise, the size of the tree is always in $\alpha+d\mathbb{Z}_+$, where $d$ and $\alpha$ are integers which depend on the offspring distribution and the type of the root (for $\alpha$ only). Particular care must thus be taken when counting the vertices of forests or specific subtrees.

We end this introduction by mentioning two papers which deal with similar limits and have appeared since the start of work. In \cite{Pen} is studied the limit of the multi-type Galton-Watson process associated to the tree, while the authors of \cite{ADG} are also interested in the local limit of the tree. The difference is that they are focused on the aperiodic case, and that they condition on the vector of population sizes of each types, and not a linear function of it. It is however shown in \cite{ADG} that, when we condition on only one type, our result can be deduced from theirs.
\medskip

The paper is split into two halves: we start by working on trees, and later on apply the results to maps. To be precise, after recalling facts about multi-type Galton-Watson trees in Section 2, we state in Section 3 the convergence of large critical multi-type Galton-Watson forests to their infinite counterpart, the proof of which is done in Section 4. Section 5 then states the basic background on planar maps, and we state and prove Theorem \ref{chap4:thcvcartes}, our main theorem of convergence of maps, in Section 6. The final section is then dedicated to an application, namely showing that the infinite Boltzmann map is almost surely a recurrent graph.

\section{Background on multi-type Galton-Watson trees}
\subsection{Basic definitions}\label{chap4:basicdef}
\smallskip
\noindent\textbf{Multi-type plane trees.} We recall the standard formalism for family trees, first introduced by Neveu in \cite{Neveu}. We denote by $\N$ the set of strictly positive integers, and $\Z_+$ the set of nonnegative integers. Let
	\[\mathcal{U}=\bigcup_{k=0}^{\infty} \N^{k}
\]
be the set of finite words on $\N$, also known as the Ulam-Harris tree. Elements of $\mathcal{U}$ are written as sequences $u=u^1u^2\ldots u^k$, and we call $|u|=k$ the height of $u$. We also let $u^-=u^1u^2\ldots u^{k-1}$ be the father of $u$ when $k>0$. In the case of the empty word $\emptyset$, we let $|\emptyset|=0$ and we do not give it a father. If $u=u^1\ldots u^k$ and $v=v^1\ldots v^l$ are two words, we define their concatenation $uv=u^1\ldots u^k v^1\ldots v^l.$

A plane tree is a subset $\mathbf{t}$ of $\mathcal{U}$ which satisfies the following conditions:
\begin{itemize}
\item $\emptyset \in \mathbf{t}$,
\item $u \in \mathbf{t}\setminus\{\emptyset\} \Rightarrow u^-\in \mathbf{t}$,
\item $\forall u\in\mathbf{t}, \exists k_u(\mathbf{t})\in\Z_+, \forall j\in\N, uj\in\mathbf{t} \Leftrightarrow j \leq k_u(\mathbf{t})$.
\end{itemize}
Given a tree $\mathbf{t}$ and an integer $n\in\Z_+$, we let $\mathbf{t}_n=\{u\in\mathbf{t},|u|= n\}$ and $\mathbf{t}_{\leq n}=\{u\in\mathbf{t},|u|\leq n\}$. We call \emph{height} of $\mathbf{t}$ the supremum $ht(t)$ of the heights of all its elements. If $u\in\mathbf{t}$, we let $\mathbf{t}_u=\{v\in\mathcal{U},uv\in\mathbf{t}\}$ be the subtree of $\mathbf{t}$ rooted at $u$.

Note that the finiteness of $k_u(\mathbf{t})$ for any vertex $u$ implies that all the trees which we consider are locally finite: a vertex can only have a finite number of neighbours. We do however allow infinite trees.

Let now $K\in\N$ be an integer. We call $[K]=\{1,2,\ldots,K\}$ the set of types. A $K$-type tree is then a pair $(\mathbf{t},\mathbf{e})$ where $\mathbf{t}$ is a plane tree and $\mathbf{e}$ is a function: $\mathbf{t}\to [K]$, which gives a type $\mathbf{e}(u)$ to every vertex $u\in\mathbf{t}$. For a vertex $u\in\mathbf{t}$, we also let $\mathbf{w}_{\mathbf{t}}(u)=\big(\mathbf{e}(u1),\ldots,\mathbf{e}(uk_u(\mathbf{t}))\big)$ be the list of types of the ordered offspring of $u$. Note of course that the knowledge of $\mathbf{e}(\emptyset)$ and of all the $\mathbf{w}_{\mathbf{t}}(u)$, $u\in\mathbf{t}$ gives us the complete type function $\mathbf{e}$.

We let
	\[\W_K =\bigcup_{n=0}^{\infty} [K]^n.
\]
be the set of finite type-lists. Given such a list $\mathbf{w}\in \W_K$ and a type $i\in[K]$, we let $p_i(\mathbf{w})=\#\{j,w_j=i\}$ and $p(\mathbf{w})=(p_i(\mathbf{w}))_{i\in[K]}$. This defines a natural projection from $\W_K$ onto $(\Z_+)^{K}$. We also let $|\mathbf{w}|=\sum_i p_i(\mathbf{w})$ be the length of $\mathbf{w}$. Elements of $\W_K$ should be seen as orderings of types, such that the type $i$ appears $p_i(\mathbf{w})$ times in the order $\mathbf{w}$.

\bigskip

\noindent\textbf{Offspring distributions.} We call \emph{ordered offspring distribution} any sequence $\z=(\zeta^{(i)})_{i\in[K]}$ where, for all $i\in[K]$, $\zeta^{(i)}$ is a probability distribution on $\W_K$. Letting for all $i$ $\mu^{(i)}=p_*\zeta^{(i)}$ be the image measure of $\zeta^{(i)}$ on $(\Z_+)^K$ by $p$, we then call $\mu=(\mu^{(i)})_{i\in[K]}$ the associated \emph{unordered} offspring distribution.

We will always assume the condition
	\[\exists i\in[K], \mu^{(i)} \left( \Big\{\mathbf{z}\in(\Z_+)^k,\sum_{j=1}^K z_j \neq 1\Big\}\right) >0
\]
to avoid degenerate cases which lead to infinite linear trees.

\bigskip

\noindent\textbf{Uniform orderings.} Let us give details about a particular case of ordered offspring distribution. For $\mathbf{n}=(n_i)_{i\in[K]}\in(\Z_+)^K$, we call uniform ordering of $\mathbf{n}$ any uniformly distributed random variable on the set of words $\mathbf{w}\in\W_K$ satisfying $p(\mathbf{w})=\mathbf{n}$. Such a random variable can be obtained by taking the word $(1,1,\ldots,1,2,\ldots,2,3,\ldots,K,\ldots,K)$ (where each $i$ is repeated $n_i$ times) and applying a uniform permutation to it. Now let $\m=(\mu^{(i)})_{i\in[K]}$ be a family of distributions on $(\Z_+)^K$, we call \emph{uniform ordering of $\m$} the ordered offspring distribution $\z=(\zeta^{(i)})_{i\in[K]}$ where, for each $i$, $\zeta^{(i)}$ is the distribution of a uniform ordering of a random variable with distribution $\mu^{(i)}$.

\bigskip

\noindent\textbf{Galton-Watson distributions.} We can now define the distribution of a $K$-type Galton-Watson tree rooted at a vertex of type $i\in[K]$ and with ordered offspring distribution $\z$, which we call $\pr^{(i)}_{\z}$, by
\begin{equation}\label{defGW}
\pr^{(i)}_{\z}(\mathbf{t},\mathbf{e})=\mathbbm{1}_{\{\mathbf{e}(\emptyset)=i\}}\prod_{u\in\mathbf{t}}\zeta^{(\mathbf{e}(u))}(\mathbf{w}_{\mathbf{t}}(u))
\end{equation}
for any finite tree $(\mathbf{t},\mathbf{e})$.
This formula only defines a sub-probability measure in general, however in the cases which interest us (namely, critical offspring distributions, see the next section) we will indeed have a probability distribution. In practice we are not interested in this formula as much as in the \emph{branching property}, which also characterizes these distributions: the types of the children of the root of a tree $(\mathbf{T,E})$ with law $\pr^{(i)}_{\z}$ are determined by a random variable with law $\zeta^{(i)}$ and, conditionally on the offspring of the root being equal to a word $\mathbf{w}$, the subtrees rooted at points $j$ with $j\in [|\mathbf{w}|]$ are independent, each one with distribution $\pr^{(j)}_{\z}$.

\bigskip

\noindent\textbf{Criticality.}\label{sec:crittree}
Let $M=(m_{i,j})_{i,j\in[K]}$ be the $K\times K$ matrix defined by
	\[m_{i,j}=\sum_{\mathbf{z}\in(\Z_+)^K} z_j \mu^{(i)}(\mathbf{z}), \qquad \forall i,j\in[K].
\]
We assume that $M$ is \emph{irreducible}, which means that, for all $i$ and $j$ in $[K]$, there exists some power $p$ such that the $(i,j)$-th entry of $M^p$ is nonzero. In this case, we know by the Perron-Frobenius theorem that the spectral radius $\rho$ of $M$ is in fact an eigenvalue of $M$. We say that $\zeta$ (or $\mu$, or $M$) is \emph{subcritical} if $\rho<1$ and \emph{critical} if $\rho=1$, which both in particular imply that Equation (\ref{defGW}) does define a probability distribution and that Galton-Watson trees with ordered offspring distribution $\zeta$ are almost surely finite. {\bf We will always assume criticality in the rest of the paper}. The Perron-Frobenius theorem also tells us that, up to multiplicative constants, the left and right eigenvectors of $M$ for $\rho$ are unique. We call them $\mathbf{a}=(a_1,\ldots,a_K)$ and $\mathbf{b}=(b_1,\ldots,b_K)$ and normalize them such that $\sum_i a_i= \sum_i a_ib_i=1$, in which case their components are all strictly positive.

The fact that $\mathbf{b}$ is a right-eigenvector of $M$ translates as 
   \[b_i=\sum_{\mathbf{z}\in(\Z_+)^{K}} \mu^{(i)}(\mathbf{z})\mathbf{b}\cdot\mathbf{z},
\] where $\cdot$ is the usual dot product. One can deduce from this the existence of a martingale naturally associated with the Galton-Watson tree. Let $(\mathbf T,\mathbf E)$ have distribution $\pr^{(i)}_{\z}$ for some $i\in[K]$ and, for all $n\in\N$ and $j\in[K]$, let $Z^{(j)}_n$ be the number of vertices of $\mathbf{T}$ which have height $n$ and type $j$, and set $\mathbf{Z}_n=(Z^{(j)}_n)_{j\in[K]}$. Define then, for $n\in\Z_{+}$,
\begin{equation}\label{defmartingale}
X_n= \mathbf{b}\cdot \mathbf{Z}_n =\sum_{j=1}^K b_j Z^{(j)}_n.
\end{equation}
The process $(X_n)_{n\in\Z_+}$ is then a martingale.

Finally, we say that $\zeta$ (or $\mu$) is \emph{regular critical} if, in addition to being critical, $\zeta$ has small exponential moments in the following sense:
	\[\exists z>1,\forall i\in[K], \sum_{\mathbf{w}\in\W_K} \zeta^{(i)}(\mathbf{w}) z^{|\mathbf{w}|} <\infty
\]

\bigskip

\noindent\textbf{Spatial trees.}
Later on in this paper we will be looking at \emph{spatial} $K$-type trees, that is trees coupled with labels on their vertices. We define a $K$-type spatial tree to be a triple $(\mathbf{t},\mathbf{e},\mathbf{l})$ where $(\mathbf{t},\mathbf{e})$ is a $K$-type tree and $\mathbf{l}$ is any real-valued function on $\mathbf{t}$. Note that, given $\mathbf{t}$, $\mathbf{e}$ and $\mathbf{l}(\emptyset)$, the rest of $\mathbf{l}$ is completely determined by the differences $\mathbf{l}(u)-\mathbf{l}(u^-)$ for $u\in\mathbf{t}\setminus\{\emptyset\}$. This is why we let, for $u\in\mathbf{t}$, $\mathbf{y}_u=\Big(\mathbf{l}(u1)-\mathbf{l}(u),\mathbf{l}(u2)-\mathbf{l}(u),\ldots,\mathbf{l}\big(uk_u(\mathbf{t})\big)-\mathbf{l}(u)\Big)\in\R^{|\mathbf{w}_{\mathbf{t}}(u)|}$ be the list of ordered label displacements of the offspring of $u$.

Consider, for all types $i\in[K]$ and words $\mathbf{w}\in\W_K$, a probability distribution $\nu^{(i)}_{\mathbf{w}}$ on $\R^{|\mathbf{w}|}$, as well as a number $\varepsilon$. We let $\pr^{(i,\varepsilon)}_{\z,\mathbf{\nu}}$ be the distribution of a triple $(\mathbf{T,E,L})$ where $(\mathbf{T,E})$ is a $K$-type tree with distribution $\pr^{(i)}_{\z}$, the root $\emptyset$ has label $\varepsilon$ and the label displacements $\big(\mathbf{L}(u1)-\mathbf{L}(u), \mathbf{L}(u2)-\mathbf{L}(u),\ldots,\mathbf{L}(uk_u(\mathbf{T}))-\mathbf{L}(u)\big)$ (with $u\in \mathbf{T}$) are all independent, each one having distribution $\nu^{\big(\mathbf{E}(u)\big)}_{\mathbf{w}_{\mathbf{T}}(u)}$ conditionally on $\mathbf{E}(u)$ and $\mathbf{w}_\mathbf{T}(u)$.

\bigskip

\noindent\textbf{Forests.}
We will not only look at trees but also at multi-type (and, when needed, labelled) \emph{forests}, a forest being defined as an ordered finite collection of trees: elements of the form $(\mathbf{f,e,l})=\big((\mathbf{t}^1,\mathbf{e}^1,\mathbf{l}^1),\ldots,(\mathbf{t}^p,\mathbf{e}^p,\mathbf{l}^p)\big)$.

A Galton-Watson random forest will be a forest where the trees are mutually independent and each one has a Galton-Watson distribution with the same ordered offspring distribution (and label increment distribution, in the labelled case). We can thus let, for $\mathbf{w}\in \W_K$, $\pr^{(\mathbf w)}_{\z}$ be the distribution of $(\mathbf{T}^i,\mathbf{E}^i)_{i\in [|\mathbf{w}|]}$ where the $(\mathbf{T}^i,\mathbf{E}^i)$ are independent, and each $(\mathbf{T}^i,\mathbf{E}^i)$ has distribution $\pr^{(w_i)}_{\z}$ and, given also a list of initial labels $\varepsilon=(\varepsilon_1,\ldots,\varepsilon_{|\mathbf{w}|})$, $\pr^{(\mathbf{w}),(\epsilon)}_{\z,\nu}$ be the distribution of $(\mathbf{T}^i,\mathbf{E}^i,\mathbf{L}^i)_{i\in [|\mathbf{w}|]}$ where the terms of the sequence are independent and, for a given $i$, $(\mathbf{T}^i,\mathbf{E}^i,\mathbf{L}^i)$ has distribution $\pr^{(w_i,\varepsilon_i)}_{\zeta,\nu}.$

All previous notation will be adapted to forests, for example, the height of a forest $\mathbf{f}$ is the maximum of the heights of its elements,  $\mathbf{f}_{\leq n}$ is the forest where each tree has been cut at height $n$, and so on.

\bigskip

\noindent\textbf{Remarks concerning notation.} For readability, we will throughout the paper use the canonical variable $(\mathbf{T},\mathbf{E})$, which is simply the identity function of the space of $K$-type trees, as well as $(\mathbf{T},\mathbf{E},\mathbf{L})$, $(\mathbf{F},\mathbf{E})$ $(\mathbf{F},\mathbf{E},\mathbf{L})$ when looking at labelled trees or forests. Thus we will, for instance, write $\pr^{(i)}_{\zeta}\big((\mathbf{T},\mathbf{E})=(\mathbf{t},\mathbf{e})\big)$ instead of $\pr^{(i)}_{\zeta}(\mathbf{t},\mathbf{e})$, for a given type $i$ and a given $K$-type tree $(\mathbf{t},\mathbf{e})$.

Moreover, since we will never change the types and labels of vertices of a tree, we will often drop $\mathbf{e}$ and $\mathbf{l}$ from the notation, once again for readability, and, in the same vein, since we only consider one offspring distribution at a time, we also often drop $\zeta$ from the $\pr_{\zeta}$ notation.

\bigskip

\noindent\textbf{Local convergence of multi-type trees and forests.} Take a sequence of $K$-type forests $(\mathbf{f}^{(n)},\mathbf{e}^{(n)})_{n\in\N}$. We say that this sequence converges locally to a $K$-type forest $(\mathbf{f},\mathbf{e})$ if, for all $k\in\N$, and $n\in\N$ large enough (depending on $k$), we have $(\mathbf{f}^{(n)}_{\leq k},\mathbf{e}^{(n)}_{\leq k})=(\mathbf{f}_{\leq k},\mathbf{e}_{\leq k})$. This convergence can be metrized: we can for example set, for two $K$-type forests $(\mathbf{f},\mathbf{e})$ and $(\mathbf{f}',\mathbf{e}')$, $d\big((\mathbf{f},\mathbf{e}),(\mathbf{f}',\mathbf{e}')\big)=\frac{1}{1+p}$ where $p$ is the supremum of all integers $k$ such that $(\mathbf{f}_{\leq k},\mathbf{e}_{\leq k})=(\mathbf{f}'_{\leq k},\mathbf{e}'_{\leq k})$.

Convergence in distribution of random forests for this metric is simply characterized: if $(\mathbf{F}^{(n)},\mathbf{E}^{(n)})_{n\in\N}$ is a sequence of random $K$-type forests, it converges in distribution to a certain random forest $(\mathbf{F},\mathbf{E})$ if and only if, for all $k\in\N$ and finite $K$-type forests $(\mathbf{f},\mathbf{e})$, the quantity $\pr\big((\mathbf{F}^{(n)}_{\leq k},\mathbf{E}^{(n)}_{\leq k})=(\mathbf{f},\mathbf{e})\big)$ converges to $\pr\big((\mathbf{F}_{\leq k},\mathbf{E}_{\leq k})=(\mathbf{f},\mathbf{e})\big).$

All these definitions can directly be adapted to the case of spatial forests: when asking for equality between the forests below height $k$, we also ask equality of the labels below this height.

\subsection{Cutting a tree at the first generation of fixed type}\label{gen1}
In this section, we fix a reference type $j\in[K]$. We are interested in the \emph{first generation of type $j$}, that is, in a $K$-type tree $\mathbf{t}$, the set of vertices of $\mathbf{t}$ with type $j$ which have no ancestors of type $j$, except maybe for the root. We then call $C_j(\mathbf{t})$ the tree formed by all the vertices which lie below or on the first generation of type $j$, including all vertices which lie on branches withno individuals of this type. If $\mathbf{T}$ has distribution $\pr^{(i)}_{\zeta}$ for some type $i$, we let $\pr^{(i)}_{\text{cut}_j}$ be the distribution of $C_j(\mathbf{T})$ and let $\mu_{i,j}$ be the distribution of the number of leaves of $C_j(\mathbf{T})$ which have type $j$ (that is, the size of the first generation of type $j$ in $\T$). Finally, if $\T$ has distribution $\pr^{(j)}$, we let $\xi_{i,j}$ be the distribution of the number of vertices of type $i$ in $C_j(\mathbf{T})$ (excluding the root, so that when $i=j$ we end up with $\xi_{j,j}=\mu_{j,j}$).

The following proposition gives a few properties of the moments of the $\mu_{i,j}$ and $\xi_{i,j}$. Most of them are already proven in \cite{M08}.

\begin{prop}\label{moments} Let $i$ and $j$ be two different types.
\begin{itemize}
\item[(i)]
	\[\sum_{k=0}^{\infty} k\mu_{i,j}(k)=\frac{b_i}{b_j}.
\]
\item[(ii)] 
	\[\sum_{k=0}^{\infty} k\xi_{i,j}(k)=\frac{a_i}{a_j}.
\]
\item[(iii)] Assume that $\zeta$ has finite second moments. Then
	\[\text{Var}(\mu_{i,i})=\frac{\sigma^2}{a_ib_i^2},
\]
where the number $\sigma> 0$ is defined by $\sigma^2=\sum_{i,j,k} a_ib_jb_kQ^{(i)}_{j,k}$, with \\ $Q^{(i)}_{j,j}=\sum_{\mathbf{z}\in(\Z_+)^K} \mu^{(i)}(\mathbf{z}) z_j(z_j-1)$ and $Q^{(i)}_{j,k}=\sum_{\mathbf{z}\in(\Z_+)^K} \mu^{(i)}(\mathbf{z}) z_jz_k$ for $j\neq k$.
\item[(iv)] Assume that $\zeta$ is regular critical. Then $\mu_{i,i}$ and $\xi_{i,j}$ also have some finite exponential moments:
	\[\exists z>1, \sum_{n\in\Z_+} \mu_{i,i}(n) z^n <\infty\text{ and } \sum_{n\in\Z_+} \xi_{i,j}(n)z^n<\infty.
\]
\end{itemize}
\end{prop}

\begin{proof}
We start with point $(i)$. We fix $j\in [K]$ and, for all $i\in [K]$, let $c_i=\sum_{k=0}^{\infty} k\mu_{i,j}(k)$. The proof that $c_i=\frac{b_i}{b_j}$ for all $i$ is done in two steps: first, show that $c_j=1$ and then that the vector $\mathbf{c}=(c_i)_{i\in[K]}$ is a right eigenvector of $M$ for the eigenvalue $1$.

The fact that $c_j=1$ is proven in \cite{M08}, Proposition 4, $(i)$. It is obtained by removing the types different from $j$ one by one, and noticing that criticality is conserved at every step until we are left with a critical monotype Galton-Watson tree.

To prove that $\mathbf{c}$ is a right eigenvector of $M$, consider a type $i\in[K]$ and apply the branching property at height $1$ in a tree with distribution $\pr^{(i)}_{\zeta}$, we get
\begin{align*}
c_i&=\sum_{\mathbf{z}\in(\Z_+)^K} \mu^{(i)}(\mathbf{z}) \Big(\sum_{l\in[K]\setminus\{j\}}  z_lc_l + z_j\Big) \\
   &= \sum_{\mathbf{z}\in(\Z_+)^K} \mu^{(i)}(\mathbf{z}) \Big(\sum_{l=1}^K z_lc_l\Big) \\
   &= \sum_{l=1}^K m_{i,l}c_l.
\end{align*}
Since $\sum_{l=1}^k m_{i,l}c_l$ is the $i$-th component of $(M\mathbf{c})$, the proof is complete.

\medskip

Point $(ii)$ was also proven in \cite{M08}, as part of the proof of Proposition 4, $(ii)$. Similarly, points $(iii)$ and $(iv)$ feature in \cite{M08}, Proposition 4.

\end{proof}




\subsection{Size of a tree and periodicity}\label{sec:per}
As said earlier, we plan on conditioning trees on being large. To do this extent, we need to define a notion of ``size" of a tree. One natural notion of size would be the total number of vertices of the tree. Another one, which, as will be shown later, is easier to work with combinatorially, would be to count only the number of vertices of one fixed type. We propose a fairly general notion of size which contains the above two examples: let $\gamma=(\gamma_1,\ldots,\gamma_K)$ be a vector of non-negative integers, one of them at least being non-zero. We then let, for a $K$-type tree $\mathbf{(t,e)}$
	\[|\mathbf{t}|_{\gamma}=\sum_{i=1}^K \gamma_i \#_i(\mathbf{t})
\]
where $\#_i(\mathbf{t})$ denotes the number of vertices of $\mathbf{t}$ with type $i\in[K]$.

One consequence of criticality is that, while a Galton-Watson tree with ordered offspring distribution $\zeta$ is almost-surely finite, the expected value of its size is infinite.

\begin{lemma}\label{expinf} For all $i\in[K]$, we have
	\[
\mathbb{E}^{(i)}\big[|\mathbf{T}|_{\gamma}\big]=\infty.
\]
\end{lemma}
\begin{proof} We just need to check that, for all $j\in[K]$, we have $\mathbb{E}^{(i)}\big[\#_j\mathbf{T}\big]=\infty.$ Let us generalize the previous section by calling, recursively, for $k>1$, the \emph{$k$-th generation of type $j$} of a tree the set of its vertices of type $j$ whose closest ancestor of type $j$ is in the $(k-1)$-th generation of type $j$. By Proposition \ref{moments}, point $(i)$, the number of vertices on each of those generations has expected value $\frac{b_j}{b_i}>0$, and thus their sum has infinite expected value.
\end{proof}
As it happens, when $\mathbf{(T,E)}$ is a Galton-Watson tree, then $|\mathbf{T}|_{\gamma}$ cannot take any integer value. For example, in a classical monotype tree, if an individual can only have an even number of children, then the total number of vertices in the tree has to be odd. Here is a precise statement for  the general case.
\begin{prop}\label{per}
There exists an integer $d\in \N$ and integers $\alpha_i$ in $\{0,1,\ldots,d-1\}$ for all types $i\in[K]$ such that, for $n\in\Z_+$:
\begin{itemize}
\item if $\pr^{(i)}(|\mathbf{T}|_{\gamma}=n)>0$, then $n\equiv \alpha_i \pmod d$.
\item if $n\equiv \alpha_i \pmod d$ and $n$ is large enough, then $\pr^{(i)}(|\T|_{\gamma}=n)>0$.
\end{itemize}
\end{prop}

\begin{rem}
This immediately extends to forests: if $\mathbf{w}\in\W_K$, then let $\alpha_{\mathbf{w}}\in\{0,\ldots,d-1\}$ such that $\alpha_{\mathbf{w}}\equiv \sum_{k=1}^{|\mathbf{w}|} \alpha_{w_k}$. Then the size in a forest with distribution $\pr^{(\mathbf{w})}_{\zeta}$ is a.s. of the form $\alpha_{\mathbf{w}}+dn$ with $n\in\Z_+$ and, if $n$ is large enough, the forest has size $\alpha_{\mathbf{w}}+dn$ with non-zero probability.
\end{rem}

The proof of Proposition \ref{per} requires the following lemma, which is a variant of the well-known ``Frobenius coin problem".
\begin{lemma}\label{coin} Let $n_1,\ldots,n_p$ be $p$ non-negative integers and let $d=gcd(n_1,\ldots,n_p)$. There exists integers $N_2,\ldots,N_p$ such that the set 
	\[\Big\{\sum_{i=1}^k k_in_i: k_1\in\Z_+ ,\; k_i\in \{0,\ldots,N_i\} \, \forall i\in\{2,\ldots,p\}\Big\}
\]
contains all large enough multiples of $d$.
\end{lemma}
The values of $N_2,\ldots,N_p$ are of no importance for us in this paper. All we need to know is that, when adding multiples of $n_1,\ldots,n_p$, if we allow all multiples of $n_1$, then we only need a finite amount of multiples of the others.

\begin{proof}
A straightforward induction shows that, if we can prove Lemma \ref{coin} in the case $p=2$, then we can generalize it to all $p$. We will thus restrict ourselves to the case where $p=2$, and can in fact further simplify the problem by dividing $n_1$ and $n_2$ by their gcd, making them coprime. In this case, $N_2$ and ``large enough" can easily be explicited: we will show that every integer greater than or equal to $n_1n_2$ can be written as $k_1n_1+k_2n_2$ with $k_1\in\Z_+$ and $k_2\in\{0,\ldots,n_1-1\}$.

Let $n\geq n_1n_2$. Since $n_1$ and $n_2$ are coprime, $n_2$ is invertible modulo $n_1$, and we know that there exists $k_2$ in $\{0,\ldots,n_1-1\}$ such that $k_2n_2 \equiv n \pmod {n_1}$. Therefore there exists $k_1\in\Z$ such that $n-k_2n_2=k_1n_1$, and since, $n\geq n_1n_2$ and $k_2 \leq n_1-1$, we also have $k_1\geq0$.
\end{proof}

Throughout the following proof, we use for $i\in[k]$ the notation $\TT^{(i)}_{\zeta}$ for the set of trees which can be obtained with positive probability starting with a root of type $i$:
	\[\TT^{(i)}_{\zeta}=\Big\{\mathbf{t}: \pr^{(i)}_{\zeta}\big(\mathbf{T}=\mathbf{t}\big)>0\Big\}
\]
\noindent\textbf{Proof of Proposition \ref{per}:} let $j\in[K]$ be a type such that $\zeta^{(j)}(\emptyset)>0$, i.e. an individual of type $j$ can die without having any children. We start by proving the existence of $d$ and $\alpha_j$ by focusing on what happens when we jump from one generation of type $j$ to the next. Let 
	\[G=\Big\{n\in\N, \pr^{(j)}(|\mathbf{T}|_{\gamma}=n)>0\Big\}.
\]
Let us also introduce some notation: if $A_1,\ldots,A_p$ are subsets of $\Z$, then we let $A_1+\ldots+A_p$ be their \emph{sumset}, that is the set of integers which can be obtained as sums $\sum_{i=1}^p a_i$ with $a_i\in A_i$ for all $i$. We also let $G^{+p}$ the p-fold iterated sumset of $G$:
	\[G^{+p}=\big\{\sum_{i=1}^p n_i: \forall i, n_i\in G\big\}
\]

The set $G$ can be obtained inductively by cutting $\T$ at its first generation of type $j$, and then grafting new trees at each vertex of this generation. To be precise, take a tree $\mathbf{t}\in \TT^{(j)}_{\zeta}$. Let $a_{\mathbf{t}}$ be the sum of all the $\gamma$-weights of all the vertices which do not have type $j$ in the cut tree $C_j(\mathbf{t})$, and let $p_{\mathbf{t}}$ be the number of vertices in the first generation of type $j$ of $\mathbf{t}$. We then have
	\[G = \bigcup_{\mathbf{t}\in\TT^{(j)}_{\zeta}} \{\gamma_j+a_{\mathbf{t}}\} + G^{+p_{\mathbf{t}}}
\]
Note of course that there is much redundance in this union, since $a_{\mathbf{t}}$ and $p_{\mathbf{t}}$ only depend on $C_j(\mathbf{t})$ ($\mathbf{t}$ up to its first generation of type $j$). Next, we do some reindexing to remove the overlap, and at the same time separate the union into three classes:
\begin{itemize}
\item the tree with only one vertex (of type $j$) is isolated in its own class.
\item the second class contains the trees with no vertices of type $j$ except for the root.
\item the third class contains all the other possible trees cut at their first generation of type $j$.
\end{itemize} 
We thus obtain
	\[G = \{\gamma_j\} \cup \bigcup_{x\in X} \{\gamma_j + a_x\} \cup \bigcup_{y\in Y} \{\gamma_j+b_y\} + G^{+p_y} 
\]
where $X$ and $Y$ are two abstract sets which we need not worry about, $a_x\in\N$ for all $x\in X$ and $(b_y,p_y)\in \Z_+\times\N$ for $y\in Y$. Note that $Y$ is non-empty (by criticality or aperiodicity).

It then follows that
	\[G=\Big\{\gamma_j + \sum_{x\in X}k_x a_x +\sum_{y\in Y} k_y (b_y+p_y\gamma_j): \sum_{x\in X}k_x \leq\sum_{y\in Y}k_yp_y  \Big\}.
\]
From this, let 
	\[d=gcd\Big(\{a_x:x\in X\} \cup \{b_y+p_y\gamma_j:y\in Y\}\Big)
\]
and let $\alpha_j\in\{0,\ldots,d-1\}$ be the remainder of $\gamma_j$ mod $d$. Then it is immediate that $G\subset \alpha_j+d\Z_+$, and Lemma \ref{coin} ensures that $G$ contains all large enough members of $\alpha_j+d\Z_+$ (the condition $\sum_{x\in X}k_x \leq\sum_{y\in Y}k_yp_y$ being weaker than the condition of Lemma \ref{coin}, we in fact only need all the $k_x$ and $k_y$ to be in a finite set, except for any one specific $k_y$).

\medskip

Now, let $i$ be a different type from $j$. We first want to show that, if $\mathbf{t}$ and $\mathbf{t'}$ are both in $\TT^{(i)}_{\zeta}$, then $|\mathbf{t}|_{\gamma}\equiv |\mathbf{t'}|_{\gamma} \pmod d$. To this end, consider a tree $\mathbf{t}^0$ in $\TT^{(j)}_{\zeta}$ which contains at least one vertex $u$ of type $i$ (such a tree exists by virtue of irreducibility). Now let $\mathbf{t}^1$ and $\mathbf{t}^2$ to be $\mathbf{t}^0$ except that we replace the subtree rooted at $u$ by, respectively, $\mathbf{t}$ and $\mathbf{t'}$. Both $\mathbf{t}^1$ and $\mathbf{t}^2$ also belong to $\TT^{(j)}_{\zeta}$, which implies that $|\mathbf{t}^1|_{\gamma}\equiv|\mathbf{t}^2|_{\gamma}\equiv \alpha_j \pmod d$, which itself implies $|\mathbf{t}|_{\gamma}\equiv|\mathbf{t'}|_{\gamma} \pmod d$. This shows the existence of $\alpha_i$.

Finally, we want to show that, if $n$ is large enough, then $\pr^{(i)} \big(|\mathbf{T}|_{\gamma}=\alpha_i+dn\big)>0$. Take any tree $\mathbf{(t,e)}\in\TT^{(i)}_{\zeta}$ which contains at least one vertex $u$ of type $j$. Let $m$ and $p$ be integers such that $|\mathbf{t}|_{\gamma}=\alpha_i+dm$ and $|\mathbf{t}_u|_{\gamma}=\alpha_j+dp$, where $\mathbf{t}_u$ is the subtree rooted at $u$. We know that, if $n$ is large enough, there exists $\mathbf{t'}$ in $\TT^{(j)}_{\zeta}$ such that $|\mathbf{t'}|_{\gamma}=\alpha_j+d(n+p-m)$. Replacing $\mathbf{t}_u$ by $\mathbf{t'}$ in $\mathbf{t}$ then yields a tree with size $\alpha_i+dn$ which itself is in $\TT^{(i)}_{\zeta}$, thus ending our proof.
\qed

\vspace{1cm}

Proposition \ref{per} can be refined in the case where we only count the number of vertices of one specific type. We leave the proof of the following corollary to the reader.
\begin{cor}\label{per1type} Assume that $\gamma_i=\mathbbm{1}_{i=1}$. Then:
\begin{itemize}
\item the period $d$ is gcd of the support of $\mu_{1,1}$, and $\alpha_1=1$.
\item for $i\in \{2,\ldots,K\}$ the measure $\mu_{i,1}$ is supported on $\alpha_i+d\Z_+$.
\end{itemize}
\end{cor}

\section{Infinite multi-type Galton-Watson trees and forests}
In this section we will consider unlabelled trees and forests with a critical ordered offspring distribution $\z$, and will omit mentioning $\z$ for readability purposes. We could in fact work with spatial trees, however, since the labellings are done conditionally on the tree and in independent fashion for each vertex, the reader can check that the proofs do not change at all if we add the labellings in.

Just as in the case of critical monotype Galton-Watson trees, multi-type trees have an infinite variant which is obtained through a size-biasing method which was first introduced in \cite{KLPP}.

\subsection{Existence of the infinite forest}

\begin{prop}\label{definftree} Let $\mathbf{w}\in\W$. There exists a unique probability measure $\widehat{\pr}^{(\mathbf{w})}$ on the space of infinite $K$-type forests such that, for any $n\in\Z_{+}$ and for any finite $K$-type forest $\mathbf{f}$ with height $n$,
\begin{equation}\label{defbiased}
\widehat{\pr}^{(\mathbf{w})}\big(\mathbf{F}_{\leq n}=\mathbf{f}\big)=\frac{1}{Z_{\mathbf{w}}}\left(\sum_{u\in \mathbf{f}_n} b_{\mathbf{e}(u)}\right) \pr^{(\mathbf{w})}\big(\mathbf{F}_{\leq n}=\mathbf{f}\big),
\end{equation}
where the normalizing constant $Z_{\mathbf{w}}$ is equal to $\sum_{i=1}^{|\mathbf{w}|} b_{w_i}=p(\mathbf{w})\cdot\mathbf{b}.$
\end{prop}


\begin{proof} Our proof is structured as the one given in \cite{LyonsPeres} for monotype trees. Let $n\in\Z_{+}$, we will first define a probability distribution $\widehat{\pr}^{(\mathbf{w})}_n$ on the space of $K$-type forests with height exactly $n$ paired with a point of height $n$. Let $(\mathbf{f},\mathbf{e})$ be such a forest and $u\in \mathbf{f}_n$, and set
	\[\widehat{\pr}^{(\mathbf{w})}_n\big(\mathbf{f},u\big)=\frac{b_{\mathbf{e}(u)}}{Z_{\mathbf{w}}} \pr^{(\mathbf{w})}\big(\mathbf{F}=\mathbf{f}\big).
\]
The martingale property of the process $(X_n)_{n\in\Z_{+}}$ defined by \eqref{defmartingale} under $\pr^{(\mathbf{w})}$ ensures us that we do have probability measures: the total mass of $\widehat{\pr}^{(\mathbf{w})}_n$ is $\frac{1}{Z_{\mathbf{w}}} \mathbb{E}^{(\mathbf{w})}_{\z}[X_n]=\frac{p(\mathbf{w})\cdot\mathbf{b}}{Z_{\mathbf{w}}}=1$.

We will check that these are compatible in the sense that, for $n\in\Z_{+}$, if $(\mathbf{F},U)$ has distribution $\widehat{\pr}^{(\mathbf{w})}_{n+1}$ then $(\mathbf{F}_{\leq n},U^-)$ has distribution $\widehat{\pr}^{(\mathbf{w})}_n$. Fix therefore $\mathbf{f}$ a $K$-type forest of height $n$ and $u$ a vertex of $t$ at height $n$. We have

\begin{align*}
\widehat{\pr}^{(\mathbf{w})}_{n+1}\big((\mathbf{F}_{\leq n},U^-)=(\mathbf{f},u)\big)&= \frac{1}{Z_{\mathbf{w}}}\pr^{(\mathbf{w})}\big(\mathbf{F}_{\leq n}=\mathbf{f}\big) \sum_{\mathbf{x}\in\W_K} \zeta^{(\mathbf{e}(u))}(\mathbf{x}) \sum_{j=1}^{|\mathbf{x}|} b_{x_j}    \\
                                                &= \frac{1}{Z_{\mathbf{w}}}\pr^{(\mathbf{w})}\big(\mathbf{F}_{\leq n}=\mathbf{f}\big) \sum_{\mathbf{z}\in(\Z_+)^K} \mu^{(\mathbf{e}(u))}(\mathbf{z})\,\mathbf{z}\cdot\mathbf{b}  \\
                                                &=\frac{1}{Z_{\mathbf{w}}}\pr^{(\mathbf{w})}\big(\mathbf{F}_{\leq n}=\mathbf{f}\big) \, b_{\mathbf{e}(u)}.
\end{align*}

Kolmogorov's consistency theorem then allows us to define a distribution $\widehat{\pr}^{(\mathbf{w})}_{\infty}$ on the set of forests where one of the trees has a distinguished infinite path. Forgetting the infinite path then gives us the distribution $\widehat{\pr}^{(\mathbf{w})}$ which we were looking for.

\end{proof}

For $n\in\Z_+$, $\mathbf{f}$ a forest of height $n+1$ and $u\in \mathbf{f}_{n+1}$, we have
\begin{align*}
\widehat{\pr}^{(\mathbf{w})}_{n+1}\Big((\mathbf{F},U)=(\mathbf{f},u) & \;|\; (\mathbf{F}_{\leq n},U^-)=(\mathbf{f}_{\leq n},u^-)\Big) =  \\
& \frac{b_{\mathbf{e(u)}}}{b_{\mathbf{e(u^-)}}} \pr^{(\mathbf{w})}_{n+1}\Big(\mathbf{F}=\mathbf{f} \; | \; \mathbf{F}_{\leq n}=\mathbf{f}_{\leq n}\Big).
\end{align*}
From this formula follows a simple description of these infinite forests.

Given a type $i\in[K]$, a random tree with distribution $\widehat{\pr}^{(i)}$ can be described in the following way: it is made of a \emph{spine}, that is an infinite ascending chain starting at the root, on which we have grafted independent trees with ordered offspring distribution $\z$. Elements of the spine have a different offspring distribution, called $\widehat{\z}$, which is a size-biased version of $\z$. It is defined by 
\begin{equation}\label{defzetahat}
\widehat{\zeta}^{(j)}(\mathbf{x})=\frac{1}{b_j}\sum_{l=1}^{|\mathbf{x}|} b_{x_l}\zeta^{(j)}(\mathbf{x}),
\end{equation}
with $j\in[K]$ and $\mathbf{x}\in \W_K$.
Given an element of the spine $u\in\mathcal{U}$ and its offspring $\mathbf{x}\in\W_K$, the probability that the next element of the spine is $uj$ for $j\in[|\mathbf{x}|]$ is proportional to $b_{x_j}$, and therefore equal to
	\[\frac{b_{x_j}}{\sum_{l=1}^{|\mathbf{x}|} b_{x_l}}.
\]

To get a forest with distribution $\widehat{\pr}^{(\mathbf{w})}$, let first $J$ be a random variable taking values in $[|\mathbf{w}|]$ such that $J=j$ with probability proportional to $b_{w_j}$. Conditionally on $J$, let $\T_J$ be a tree with distribution $\widehat{\pr}^{(J)}$, and let $\mathbf{T}_i$, for $i\in[|\mathbf{w}|]$, $i\neq J$ be a tree with distribution $\pr^{(i)}$, all these trees being mutually independent. Then the forest $(\mathbf{T}_i)_{i\in [|\mathbf{w}|]}$ has distribution $\widehat{\pr}^{(\mathbf{w})}$.

\begin{rem} Recall that a tree with law $\pr^{(i)}$ is finite for any $i\in[K]$. Therefore, a forest with distribution $\widehat{\pr}^{(\mathbf{w})}$ can only have one infinite path, and thus we do not lose any information by going from $\widehat{\pr}^{(\mathbf{w})}_{\infty}$ to $\widehat{\pr}^{(\mathbf{w})}$.
\end{rem}

\subsection{Convergence to the infinite forest}
Recall from Section \ref{sec:per} the notations $d$ and $\alpha_{\mathbf{w}}$: the size of a forest with distribution $\pr^{(\mathbf{w})}_{\zeta}$ is always of the form $\alpha_{\mathbf{w}}+dn$.
\begin{theo}\label{cvforests}
Assume one of the following:
\begin{itemize}
\item $\gamma_j=\mathbbm{1}_{j=1}$ for $j\in[K]$.
\item $\zeta$ is regular critical.
\end{itemize}

As $n$ tends to infinity, a forest $\mathbf{F}$ with distribution $\pr^{(\mathbf{w})}$, conditioned on $|\mathbf{F}|_{\gamma}=\alpha_{\mathbf{w}}+dn$, converges in distribution to a forest with distribution $\widehat{\pr}^{(\mathbf{w})}$. In other words, given a forest $(\mathbf{f},\mathbf{e})$ of height $k$, we have
	\[\pr^{(\mathbf{w})}\big(\mathbf{F}_{\leq k}=\mathbf{f} \mid |\mathbf{F}|_{\gamma}=\alpha_{\mathbf{w}}+dn\big) \underset{n\to\infty} \longrightarrow \widehat{\pr}^{(\mathbf{w})}\big(\mathbf{F}_{\leq k}=\mathbf{f}\big)
\]
\end{theo}

This theorem is split into two quite distinct parts. For the first part, we assume that the notion of size of a tree we take is simply the amount of vertices of one fixed type, which we can take as $1$ by symmetry. In this case, the theorem will be proved with purely combinatorial tools, notably ratio limit theorems for random walks. In the second part, we do not make any assumptions on $\gamma$, and in exchange for that we have to restrict ourselves to the case where the offspring distribution has exponential moments. The result will then be proved with the help of techniques from analytic combinatorics.

\section{Proof of Theorem \ref{cvforests}}
\subsection{The main ingredient}
Whether we count only one type of vertex or the offspring distribution is regular critical, the proof of Theorem \ref{cvforests} will rely on the following asymptotic equivalence, indexed by any word $\mathbf{w}\in\W_K$:

\begin{equation}\tag{$H_{\mathbf{w}}$}\label{byword}
\pr^{(\mathbf{w})}(|\mathbf{F}|_{\gamma}=\alpha_{\mathbf{w}}+dn) \underset{n\to\infty}{\sim} \frac{Z_{\mathbf{w}}}{b_1} \pr^{(1)} \big(|\mathbf{T}|_{\gamma}=\alpha_1+d(n+p)\big), \qquad \forall p\in\Z
\end{equation}

What Equation $\eqref{byword}$ means is that, when we ask for a forest to have size of order $dn$ with large $n$, then exactly one of its tree components will have size or order $dn$, while the others will be comparatively microscopic.

\medskip

\noindent\textbf{Proof that Theorem \ref{cvforests} follows from \eqref{byword}:} take a $K$-type forest $\mathbf{f}$ with height $k\in\N$, and let $\mathbf{x}\in\W_K$ be the word obtained by taking the types of the vertices of $\mathbf{f}$ with height $k$ (the order of the elements $\mathbf{x}$ actually has no influence). For $n$ large enough, we have
\begin{align*}\pr^{(\mathbf{w})}\big(\mathbf{F}_{\leq k}=\mathbf{f} \mid |\mathbf{F}|_{\gamma}=\alpha_{\mathbf{w}}+dn\big)
	&=\frac{\pr^{(\mathbf{w})}\big(\mathbf{F}_{\leq k}=\mathbf{f},|\mathbf{F}|_{\gamma}=\alpha_{\mathbf{w}}+dn \big)}{\pr^{(\mathbf{w})}(|\mathbf{F}|_{\gamma}=\alpha_{\mathbf{w}}+dn)} \\
	&=\pr^{(\mathbf{w})}\big(\mathbf{F}_{\leq k}=\mathbf{f}\big)\frac{\pr^{(\mathbf{x})}(|\mathbf{F}|_{\gamma}=\alpha_{\mathbf{w}}+dn-q)}{\pr^{(\mathbf{w})}(|\mathbf{F}|_{\gamma}=\alpha_{\mathbf{w}}+dn)}
\end{align*}
where $q=|\mathbf{f}_{\leq k-1}|_{\gamma}$. By the results of Section \ref{sec:per}, if $\pr^{(\mathbf{w})}\big(\mathbf{F}_{\leq k}=\mathbf{f}\big)>0$ then $\alpha_{\mathbf{w}}-q$ must be congruent to $\alpha_{\mathbf{x}}$ modulo $d$, giving us
\begin{align*}
\pr^{(\mathbf{w})}\big(\mathbf{F}_{\leq k}=\mathbf{f} \mid |\mathbf{F}|_{\gamma}= \alpha_{\mathbf{w}}+dn\big)  
     = 	\pr^{(\mathbf{w})}\big(\mathbf{F}_{\leq k}=\mathbf{f}\big)\frac{\pr^{(\mathbf{x})}\big(|\mathbf{F}|_{\gamma}=\alpha_{\mathbf{x}}+d(n+p)\big)}{\pr^{(\mathbf{w})}\big(|\mathbf{F}|_{\gamma}=\alpha_{\mathbf{w}}+dn\big)}
\end{align*}
for some signed integer p. Now if we let $n$ tend to infinity, using both \eqref{byword} and $(H_{\mathbf{x}})$, we obtain
	\[\frac{\pr^{(\mathbf{x})}\big(|\mathbf{F}|_{\gamma}=\alpha_{\mathbf{x}}+d(n+p)\big)}{\pr^{(\mathbf{w})}\big(|\mathbf{F}|_{\gamma}=\alpha_{\mathbf{w}}+dn\big)}
	\underset{n\to\infty}{\longrightarrow}\frac{Z_{\mathbf{x}}}{Z_{\mathbf{w}}}=\frac{1}{Z_{\mathbf{w}}}\left(\sum_{u\in \mathbf{f}_k} b_{\mathbf{e}(u)}\right),
\]
which concludes the proof of Theorem \ref{cvforests}, assuming $\eqref{byword}$.
\qed

\subsection{Proving \eqref{byword} when counting only one type}
We assume from now on that $\gamma_j=\mathbbm{1}_{j=1}$ for all $j\in[K]$, and will therefore from now on write $\#_1 \mathbf{T}$ for $|\mathbf{T}|_{\gamma}$. Recall from Section \ref{sec:per} in particular that $d$ is the gcd of the support of $\mu_{1,1}$ and that $\alpha_1=1$.

Obtaining \eqref{byword} for every word $\mathbf{w}$ will be done in several small steps. We will first prove it for some fairly simple words and gradually enlarge the class of $\mathbf{w}$ for which it holds, until we have every element of $\W_K$.
\subsubsection{Ratio limit theorems for a random walk}
\label{sec:ratio}
Let $(S_n)_{n\in\N}$ be a random walk which starts at $0$ and whose jumps are all greater than or equal to $-1$, their distribution being given by $\pr(S_1=k)=\mu_{1,1}(k+1)$ for $k\geq -1$.

\begin{lemma}\label{ratio} For all $\alpha\in \{0,\ldots,d-1\}$ we have
	\[\pr(S_{\alpha+dn}=-\alpha) \underset{n\to\infty}{\sim} \pr(S_{dn}=0) \underset{n\to\infty}{\sim} \pr(S_{d(n+1)}=0)
\]
\end{lemma}

\begin{proof} The first thing to notice is that the random walk $(\frac{S_{dn}}{d})_{n\in\N}$ is irreducible, recurrent and aperiodic on $\Z$. First, it is indeed integer-valued because, by definition, for every $n$, $S_{n+1}\equiv S_n-1\pmod d$, and thus we stay in the same class modulo $d$ if we take $d$ steps at a time. Irreducibility comes from the fact that steps of $(S_n)_{n\in\N}$ has a nonzero probability of being equal to $-1$ because $\mu_{j,j}(0)>0$, and thus $(\frac{S_{dn}}{d})_{n\in\N}$ can have positive jumps or jumps equal to $-1$. Since the jumps of $(S_n)_{n\in\N}$ are centered by Proposition \ref{moments}, point $(i)$ this makes $(\frac{S_{dn}}{d})_{n\in\N}$ an irreducible and centered random walk on $\Z$, so that it is recurrent (see for example Theorem 8.2 in \cite{Kallenberg}). Finally, aperiodicity is obtained from the fact that, if $\mu_{j,j}(n)>0$, then $\pr(S_n=0)>0$ by jumping straight to $n-1$ and going down to $0$ one step at a time.

As a consequence of this, we can apply Spitzer's strong ratio theorem (see \cite{Spitzer}, p.49) to the random walk $(\frac{S_{dn}}{d})_{n\in\N}$. We obtain that, for any $k\in\Z$,
	\[\pr(S_{dn}=0) \underset{n\to\infty}\sim\pr(S_{d(n+1)}=0) \underset{n\to\infty}\sim \pr(S_{dn}=dk).
\]
This proves the second half of Lemma \ref{ratio}, and can also be used to prove the first half. Let $\mu_{j,j}^{* \alpha}$ be the distribution of the sum of $\alpha$ independent variables with distribution $\mu_{j,j}$. For $n\in\N$, we then have
	\[\pr(S_{\alpha+dn}=-\alpha)=\sum_{p\in\Z}\pr(S_{dn}=-\alpha-p)\mu^{* \alpha}_{j,j} (p+\alpha).
\]
Fatou's lemma then gives us
	\[\underset{n\to\infty}\liminf \ \frac{\pr(S_{\alpha+dn}=-\alpha)}{\pr(S_{dn}=0)}\geq \sum_{p\in\Z}\mu^{* \alpha}_{j,j} (p+\alpha)=1
\]
A similar argument also shows that
	\[\underset{n\to\infty}\liminf \ \frac{\pr(S_{d(n+1)}=0)}{\pr(S_{\alpha+dn}=-\alpha)} \geq 1,
\]
and this ends the proof.
\end{proof}

\subsubsection{The case where $\mathbf{w}=(1,1,\ldots,1)$}\label{easycase}

Consider a tree $\T$ with distribution $\pr^{(1)}$. Consider then the reduced tree $\Pi^{(1)}(\T)$ where all the vertices with types different from $1$ have been erased but ancestral lines are kept (such that the father of a vertex of $\Pi^{(1)}(T)$ is its closest ancestor of type $1$ in $\mathbf{T}$). This tree is precisely studied in \cite{M08}, where it is shown that it is a monotype Galton-Watson tree, its offspring distribution naturally being $\mu_{1,1}$. As a result, the well-known \emph{cyclic lemma} (see \cite{PitmanStFl}, Sections 6.1 and 6.2) tells us that
  \[\pr^{(1)}(\#_1 \mathbf{T}=1+dn)=\frac{1}{1+dn}\pr(S_{1+dn}=-1).
\]
where $(S_n)_{n\in\N}$ is the random walk defined in Section \ref{sec:ratio}. One particular consequence of this is the fact that, thanks to Lemma \ref{ratio}, in order to prove \eqref{byword} for a certain word $\mathbf{w}$, we can restrict ourselves to proving the asymptotic equivalence for a single value of $p$, which will we take to be $0$.

Consider now a word $\mathbf{w}=(1,1,\ldots,1)$ of length $k$, where $k$ is any integer. The cyclic lemma can be adapted to forests (see \cite{PitmanStFl} again), and we have
	\[\pr^{(\mathbf{w})}(\#_1 \mathbf{F}=k+dn)=\frac{k}{k+dn}\pr(S_{k+dn}=-k),
\]
Lemma \ref{ratio} then implies \eqref{byword} in this case since $Z_{\mathbf{w}}=kb_1$ and $\alpha_{\mathbf{w}}=k$.

The cases where $\mathbf{w}$ contains types different from $1$ will be much less simple, and we first start with an inequality.

\subsubsection{A lower bound for general $\mathbf{w}$}\label{lowerbound}

Let $\mathbf{w}\in\W_K$. In order to count the number of vertices of type $1$ of a forest with distribution $\pr^{(\mathbf{w})}$, we cut it at its first generation of type $1$. 
	\[\pr^{(\mathbf{w})}(\#_1\mathbf{F}=\alpha_{\mathbf{w}}+dn)=\sum_{i=1}^{|\mathbf{w}|}\sum_{k_i=0}^{\infty}\mu_{{w_i},1}(k_i)\pr^{(1,\ldots,1)}(\#_1\mathbf{F}=\alpha_{\mathbf{w}}-q+dn)
\]
where $q$ is the number of times $1$ appears in $\mathbf{w}$ and $1$ is repeated $k_1+k_2,\ldots+k_{|\mathbf{w}|}$ times in $\pr^{(1,\ldots,1)}$. By Corollary \ref{per1type}, whenever $\mu_{{w_i},1}(k_i)>0$, we have $\alpha_{w_i}\equiv k_i-\mathbbm{1}_{w_i=1} \pmod d$, and thus the use of $H_{(1,\ldots,1)}$, combined with Fatou's lemma, gives us the following lower bound:
	\[\underset{n\to\infty}\liminf\,\frac{\pr^{(\mathbf{w})}(\#_1\mathbf{F}=\alpha_{\mathbf{w}}+dn)}{\pr^{(1)} (\#_1\mathbf{T}=1+dn)} \geq \sum_{i=1}^{|\mathbf{w}|}\sum_{k_i}k_i\mu_{w_{k_i},1}(k_i).
\]
We can then use point $(i)$ of Proposition \ref{moments} to identify the right-hand side and obtain

\begin{equation}\label{eq:lowerbound}
\underset{n\to\infty}\liminf\,\frac{\pr^{(\mathbf{w})}(\#_1\mathbf{F}=\alpha_{\mathbf{w}}+dn)}{\pr^{(1)} (\#_1\mathbf{T}=1+dn)} \geq \frac{Z_\mathbf{w}}{b_1}.
\end{equation}

To prove the reverse inequality for the limsup, we will try to fit a forest with distribution $\pr^{(\mathbf{w})}$ ``inside" a tree with distribution $\pr^{(1)}$. We first need some additional notions.

\subsubsection{The extension relation}
We describe here a tool which will be useful in the future. Let $(\mathbf{t},\mathbf{e})$ and $(\mathbf{t'},\mathbf{e'})$ be two $K$-type trees. We say that $\mathbf{t}'$ \emph{extends} $\mathbf{t}$, which we write $\mathbf{t}'\vdash\mathbf{t}$ (omitting as usual the type functions for clarity) if $\mathbf{t}'$ can be obtained from $\mathbf{t}$ by grafting trees on the leaves of $\mathbf{t}'$. More precisely, $\mathbf{t}'\vdash\mathbf{t}$ if:
\begin{itemize}
\item $\mathbf{t}\subset\mathbf{t}'$.
\item $\forall u\in\mathbf{t}$, $\mathbf{e}(u)=\mathbf{e'}(u)$.
\item $\forall u\in\mathbf{t}'\setminus\mathbf{t},\exists v\in\partial\mathbf{t},w\in\mathcal{U}: \;u=vw$.
\end{itemize}
Here, $\partial\mathbf{t}$ is the set of leaves of $\mathbf{t}$, that is the set of vertices $v$ of $\mathbf{t}$ such that $k_v(\mathbf{t})=0$.

\begin{figure}[ht]
\centering
\begingroup%
  \makeatletter%
  \providecommand\color[2][]{%
    \errmessage{(Inkscape) Color is used for the text in Inkscape, but the package 'color.sty' is not loaded}%
    \renewcommand\color[2][]{}%
  }%
  \providecommand\transparent[1]{%
    \errmessage{(Inkscape) Transparency is used (non-zero) for the text in Inkscape, but the package 'transparent.sty' is not loaded}%
    \renewcommand\transparent[1]{}%
  }%
  \providecommand\rotatebox[2]{#2}%
  \ifx\svgwidth\undefined%
    \setlength{\unitlength}{194.83195893bp}%
    \ifx\svgscale\undefined%
      \relax%
    \else%
      \setlength{\unitlength}{\unitlength * \real{\svgscale}}%
    \fi%
  \else%
    \setlength{\unitlength}{\svgwidth}%
  \fi%
  \global\let\svgwidth\undefined%
  \global\let\svgscale\undefined%
  \makeatother%
  \begin{picture}(1,0.96487871)%
    \put(0,0){\includegraphics[width=\unitlength]{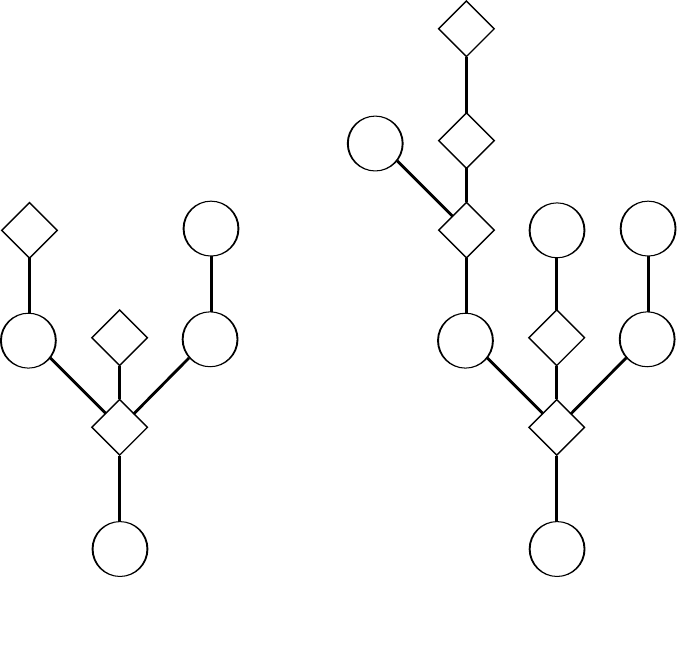}}%
    \put(0.1559672,0.00518476){\color[rgb]{0,0,0}\makebox(0,0)[lb]{\smash{$\mathbf{t}$}}}%
    \put(0.80779869,0.00518488){\color[rgb]{0,0,0}\makebox(0,0)[lb]{\smash{$\mathbf{t'}$}}}%
  \end{picture}%
\endgroup%
\caption{An example of a $2$-type tree extending another. Here, $\mathbf{t}'\vdash\mathbf{t}$.}
\end{figure}
This is once again adaptable to forests: if $(\mathbf{f},\mathbf{e})$ and $(\mathbf{f'},\mathbf{e'})$ are two $k$-type forests, then we say that $\mathbf{f}'\vdash\mathbf{f}$ if they have the same number of tree components and each tree of $\mathbf{f}'$ extends the corresponding tree of $\mathbf{f}$.

The extension relation behaves well with Galton-Watson random forests. For example, the following is immediate from the branching property:
\begin{lemma} If $(\mathbf{f,e})$ is a finite forest and $\mathbf{w}$ the list of types of the roots of its components, then
	\[\pr^{(\mathbf{w})}(\mathbf{F}\vdash\mathbf{f})=\prod_{u\in\mathbf{f}\setminus\partial\mathbf{f}}\zeta^{(\mathbf{e}(u))}(\mathbf{w}_{\mathbf{f}}(u))
\]
Moreover, we have a generalization of the branching property: conditionally on $\mathbf{F}\vdash\mathbf{f}$, $\mathbf{F}$ is obtained by appending independent trees at the leaves of $\mathbf{f}$, and for every such leaf $v$, the tree grafted at $v$ has distribution $\pr^{(\mathbf{e}(v))}$.
\end{lemma}

For infinite trees, we get a generalization of \eqref{defbiased}:
\begin{lemma}\label{extensionbiased} If $(\mathbf{f,e})$ is a finite forest, let $\mathbf{x}$ be the word formed by the types of the leaves of $\mathbf{f}$ in lexicographical order. We have
	\[\widehat\pr^{(\mathbf{w})}(\mathbf{F}\vdash\mathbf{f})=\frac{Z_{\mathbf{x}}}{Z_{\mathbf{w}}}\pr^{(\mathbf{w})}(\mathbf{F}\vdash\mathbf{f}).
\]
\end{lemma}

\begin{proof}
Let $n$ be the height of $\mathbf{f}$. Any forest of height $n$ which extends $\mathbf{f}$ can be obtained by adding after each leaf $u$ of $\mathbf{f}$ a tree with height smaller than $n-|u|$. Let $u_1,\ldots,u_p$ be the leaves of $\mathbf{f}$, and $e_1,\ldots,e_p$ be their types, we will append for all $i$ a tree $(\mathbf{t}^i,\mathbf{e}^i)$ to the leaf $u_i$ and call the resulting forest $(\tilde{\mathbf{f}},\tilde{\mathbf{e}})$, implicitly a function of $\mathbf{f}$ and $\mathbf{t}^1,\ldots,\mathbf{t}^p$. Thus, recalling the notation $X$ for the martingale defined in Equation \eqref{defmartingale},
\begin{align*}
\widehat\pr^{(\mathbf{w})}(\mathbf{F}\vdash\mathbf{f})&=\sum_{\mathbf{t}_1,\ldots,\mathbf{t}_p} \sum_{v\in\tilde{\mathbf{f}}_n}\frac{ b_{\tilde{\mathbf{e}}(v)}}{Z_{\mathbf{w}}} \pr^{(\mathbf{w})}\big((\mathbf{F}_{\leq n},\mathbf{E}_{\leq n})=(\tilde{\mathbf{f}},\tilde{\mathbf{e}})\big)\\
&=\sum_{\mathbf{t}_1,\ldots,\mathbf{t}_p} \sum_{i=1}^p\sum_{v\in\mathbf{t}^i_{n-|u_i|}}\frac{ b_{\tilde{\mathbf{e}}(v)}}{Z_{\mathbf{w}}}\pr^{(\mathbf{w})}(\mathbf{F}\vdash\mathbf{f})\prod_{i=1}^p\pr^{(e_i)}(\mathbf{T}_{\leq n-|u_i|}=\mathbf{t}_i)\\
&=\frac{\pr^{(\mathbf{w})}(\mathbf{F}\vdash\mathbf{f})}{Z_{\mathbf{w}}}\sum_{i=1}^p \sum_{\mathbf{t}_1,\ldots,\mathbf{t}_p}\sum_{v\in\mathbf{t}^i_{n-|u_i|}}b_{\mathbf{e}_i(v)}\prod_{i=1}^p\pr^{(e_i)}(\mathbf{T}_{\leq n-|u_i|}=\mathbf{t}_i)\\
&=\frac{\pr^{(\mathbf{w})}(\mathbf{F}\vdash\mathbf{f})}{Z_{\mathbf{w}}}\sum_{i=1}^p \mathbb{E}^{(e_i)}[X_{n-|u_i|}] \\
&=\frac{\pr^{(\mathbf{w})}(\mathbf{F}\vdash\mathbf{f})}{Z_{\mathbf{w}}}\sum_{i=1}^p b_{e_i} \\
&=\frac{Z_{\mathbf{x}}}{Z_{\mathbf{w}}}\pr^{(\mathbf{w})}(\mathbf{F}\vdash\mathbf{f}).
\end{align*}

\end{proof}

\subsubsection{The case where there is a tree $\mathbf{t}$ such that $\pr^{(1)}(\mathbf{T}\vdash\mathbf{t})>0$ and $\mathbf{w}$ is the word formed by the leaves of $\mathbf{t}$}\label{medium}
Let $(\mathbf{t},\mathbf{e})$ be a tree with root of type $1$ such that $\pr^{(1)}(\mathbf{T}\vdash\mathbf{t})>0$. Let $\mathbf{w}$ be the word formed by the types of the leaves of $\mathbf{t}$, we will prove \eqref{byword}. We first need an intermediate lemma.

\begin{lemma}\label{completion} There exists a countable family of trees $\mathbf{t}^{(2)},\mathbf{t}^{(3)}\ldots$ such that, for any $K$-type tree $\mathbf{t}'$ with root of type $1$:
\begin{itemize}
\item either $\mathbf{t}\vdash \mathbf{t}'$.
\item or $\mathbf{t}'\vdash \mathbf{t}$.
\item or there is a unique $i$ such that $\mathbf{t}'\vdash\mathbf{t}^{(i)}$.
\end{itemize}
\end{lemma}

\begin{proof}
For all $k\in \{2,3,\ldots,ht(\mathbf{t})\}$, take all the trees $\mathbf{t'}$ which have height $k$ and which satisfy both $\mathbf{t'}_{\leq k-1}=\mathbf{t}_{\leq k-1}$ and $\mathbf{t'}_k\neq\mathbf{t}_k.$ These are in countable amount and we can therefore call them $(\mathbf{t}^{(i)})_{i\geq 2}$ in any order. Now for any $K$-type tree $\mathbf{t'}$ with root of type $1$, by considering the highest integer $k$ such that $\mathbf{t'}_{\leq k-1}=\mathbf{t}_{\leq k-1}$, we directly obtain that, if none of $\mathbf{t}$ and $\mathbf{t'}$ extend the other, then $\mathbf{t}'$ extends one of the $\mathbf{t}^{(i)}$.
\end{proof}

Now let $\mathbf{t}^{(1)}=\mathbf{t}$, and, for all $i\in\N$, let also $\mathbf{w}^i$ be the word formed by the types of the leaves of $\mathbf{t}^{(i)}$. Write
\begin{align*}
\pr^{(1)}(\#_1 \mathbf{T}=1+dn)&=\sum_{i=1}^{\infty}\pr^{(1)}(\mathbf{T}\vdash \mathbf{t}^{(i)},\#_1 \mathbf{T}=1+dn) + \pr^{(1)}( \mathbf{t}\vdash\mathbf{T},\mathbf{t}\neq\mathbf{T},\#_1 \mathbf{T}=1+dn)\\
                    &=\sum_{i=1}^{\infty}\pr^{(1)}(\mathbf{T}\vdash \mathbf{t}^{(i)})\pr^{(\mathbf{w}^i)}(\#_1\mathbf{F}=1-q^{(i)}+dn) + \pr^{(1)}( \mathbf{t}\vdash\mathbf{T},\mathbf{t}\neq\mathbf{T},\#_1 \mathbf{T}=1+dn)
\end{align*}
where $q^{(i)}$ is the number of vertices of type $1$ of $\mathbf{t}^{(i)}$ which are not leaves.
Divide by $\pr^{(1)}(\#_1 \mathbf{T}=1+dn)$ on both sides of the equation to obtain
\begin{equation}
\label{eqdansmedium}\sum_{i=1}^{\infty}\pr^{(1)}(\mathbf{T}\vdash \mathbf{t}^{(i)})\frac{\pr^{(\mathbf{w}^i)}(\#_1\mathbf{F}=1-q^{(i)}+dn)}{\pr^{(1)}(\#_1 \mathbf{T}=1+dn)}+\pr^{(1)}( \mathbf{t}\vdash\mathbf{T}\;| \;\#_1 \mathbf{T}=1+dn) =1
\end{equation}
Note that 
	\[\pr^{(1)}( \mathbf{t}\vdash\mathbf{T}\;| \;\#_1 \mathbf{T}=1+dn)
\]
is equal to $0$ for $n$ large enough, since $\mathbf{t}$ is finite.

By the results of Section \ref{sec:per}, we have $1-q^{(i)}\equiv \alpha_{\mathbf{w}^{(i)}} \pmod d$ for all $i\in\N$, and thus, using the lower bound \eqref{eq:lowerbound}, we have
	\[\underset{n\to\infty}\liminf\, \pr^{(1)}(\mathbf{T}\vdash \mathbf{t}^{(i)})\frac{\pr^{(\mathbf{w}^i)}(\#_1\mathbf{F}=1-q^{(i)}+dn)}{\pr^{(1)}(\#_1 \mathbf{T}=1+dn)} \geq \pr^{(1)}(\mathbf{T}\vdash \mathbf{t}^{(i)})\frac{Z_{\mathbf{w}^i}}{b_1}
\] for all $i\in\N$. However, by Lemma \ref{extensionbiased} and Lemma \ref{completion}, we have
	\[\sum_{i=1}^{\infty} \pr^{(1)}(\mathbf{T}\vdash \mathbf{t}^{(i)})\frac{Z_{\mathbf{w}^i}}{b_1} = \sum_{i=1}^{\infty} \widehat{\pr}^{(i)}(\mathbf{T}\vdash \mathbf{t}^{(i)}) =1,
\]
and thus, whenever $\pr^{(1)}(\mathbf{T}\vdash \mathbf{t}^{(i)})$ is nonzero, we must have 
	\[\underset{n\to\infty} \limsup\,\frac{\pr^{(\mathbf{w}^i)}(\#_1\mathbf{F}=1-q^{(i)}+dn)}{\pr^{(1)}(\#_1 \mathbf{T}=1+dn)} \leq \frac{Z_{\mathbf{w}^i}}{b_1},
\]
which ends the proof of \eqref{byword}.

\subsubsection{Removing one element from $\mathbf{w}$}
\begin{lemma}\label{reduce} Let $\mathbf{w}\in\W_K$ be such that \eqref{byword} holds. Let $m$ be any integer in $[|\mathbf{w}|]$ and let $\tilde{\mathbf{w}}$ be $\mathbf{w}$, except that we remove $w_m$ from the list. Then $(H_{\tilde{\mathbf{w}}})$ also holds.
\end{lemma}

\begin{proof} For $n\in\N$, we split the event $\{\#_1\mathbf{F}=\alpha_{\mathbf{w}}+dn\}$ according to the first and second generations of type $1$ in the $m$-th tree of the forest. By calling $k$ the number of vertices in the first generation of type $1$ issued from the $m$-th tree, and then $i_1,\ldots,i_k$ the numbers of vertices in the first generation of type $1$ of each corresponding subtree, we have	\[\pr^{(\mathbf{w})}(\#_1\mathbf{F}=\alpha_{\mathbf{w}}+dn)=\sum_{k}\mu_{w_m,1}(k)\sum_{i_1,\ldots,i_k}\prod_{r=1}^k\mu_{1,1}(i_r)\pr^{(\tilde{\mathbf{w}}^{i_1+\ldots+i_r})}(\#_1\mathbf{F}=\alpha_{\mathbf{w}}-k-\mathbbm{1}_{\{w_m=1\}}+dn)
\]
where $\tilde{\mathbf{w}}^{i_1+\ldots+i_r}$ is the word $\mathbf{w}$ where $w_m$ has been replaced by $1$, repeated $i_1+\ldots+i_r$ times. Note that the term of the sum where $k=0$ is to be interpreted as $\pr^{(\tilde{\mathbf{w}})}(\#_1\mathbf{F}=\alpha_{\mathbf{w}}-\mathbbm{1}_{\{w_m=1\}}+dn)$.

We now use the same argument as in the end of the previous section: we first divide by $\pr^{({\mathbf{w}})}(\#_1\mathbf{F}=\alpha_{\mathbf{w}}+dn)$ to get
	\[\sum_{k}\mu_{w_m,1}(k)\sum_{i_1,\ldots,i_k}\prod_{r=1}^k\mu_{1,1}(i_r)\frac{\pr^{(\tilde{\mathbf{w}}^r)}(\#_1\mathbf{F}=\alpha_{\mathbf{w}}-k-\mathbbm{1}_{\{w_m=1\}}+dn)}{\pr^{({\mathbf{w}})}(\#_1\mathbf{F}=\alpha_{\mathbf{w}}+dn)}=1.
\]
For each choice of $k$ and $i_1,\ldots,i_k$, using lower bound \eqref{eq:lowerbound} as well as \eqref{byword}, we have
	\[\underset{n\to\infty}\liminf \mu_{w_m,1}(k)\prod_{r=1}^k\mu_{1,1}(i_r)\frac{\pr^{(\tilde{\mathbf{w}}^r)}(\#_1\mathbf{F}=\alpha_{\mathbf{w}}-k-\mathbbm{1}_{\{w_m=1\}}+dn)}{\pr^{({\mathbf{w}})}(\#_1\mathbf{F}=\alpha_{\mathbf{w}}+dn)} \geq \mu_{w_m,1}(k)\prod_{r=1}^k\mu_{1,1}(i_r)\frac{Z_{\tilde{\mathbf{w}}}+\sum_{r} i_r}{Z_{\mathbf{w}}}.
\]
A repeated use of point $(i)$ of Proposition \ref{moments} shows that these add up to $1$, and thus, for $k$ and $i_1,\ldots,i_k$ such that $\mu_{w_m,1}(k)\prod_{r=1}^k\mu_{1,1}(i_r)\neq 0$, we do have
	\[\underset{n\to\infty}\lim \frac{\pr^{(\tilde{\mathbf{w}}^r)}(\#_1\mathbf{F}=\alpha_{\mathbf{w}}-k-\mathbbm{1}_{\{w_m=1\}}+dn)}{\pr^{(\tilde{\mathbf{w}})}(\#_1\mathbf{F}=\alpha_{\mathbf{w}}+dn)}=\frac{Z_{\tilde{\mathbf{w}}}+\sum_{r=1}^k i_r}{Z_{\mathbf{w}}}.
\]
By irreducibility, one can find $k$ such that $\mu_{w_m,1}(k)\neq0$, and by criticality one has $\mu_{1,1}(0)\neq 0$, meaning that we can take $i_1,\ldots,i_k$ all equal to zero, and this ends the proof.
\end{proof}

\subsubsection{End of the proof}
By applying Lemma \ref{reduce} repeatedly and using the fact that \eqref{byword} stays true if we permute the terms of $\mathbf{w}$, we obtain that, if $\mathbf{w}$ and $\mathbf{w}'$ are two words such that any type features fewer times in $\mathbf{w}'$ than in $\mathbf{w}$, then \eqref{byword} implies $(H_{\mathbf{w}'})$. Thus, by Section \ref{medium}, we now only need to show the following lemma.
\begin{lemma} For all nonnegative integers $n_1,\ldots,n_K$, there exists a $K$-type tree $(\mathbf{t},\mathbf{e})$ which has more than $n_i$ leaves of type $i$ for all $i\in[K]$, and such that $\pr^{(1)}(\mathbf{T}\vdash\mathbf{t})>0$.
\end{lemma}
\begin{proof}
The first step is showing that, for $p$ large enough, the $p$-th generation of type $1$ of $\mathbf{T}$ has positive probability of having more than $n_1+\ldots+n_K$ vertices, where the $p$-th generation of type $1$ is the set of vertices of type $1$ which have exactly $p$ ancestors of type $1$ including the root. This is immediate because the average of $\mu_{1,1}$ is $1$ and we are not in a degenerate tree, and thus the size of each generation of type $1$ has positive probability of being strictly larger than the previous generation.

Irreducibility then tells us that, after each vertex of the $p$-th generation of type $1$, there is a positive probability of finding a vertex of type $i$ for any $i$.
\end{proof}

\subsection{Proving \eqref{byword} when $\zeta$ is regular critical}
We now take general $\gamma$ and assume that $\zeta$ is regular critical. Our aim here is to prove the following refinement of \eqref{byword}: there exists a constant $C>0$ such that, for all $\mathbf{w}\in\W_K$
\begin{equation}\tag{$H'_{\mathbf{w}}$}\label{bywordprime}
\pr^{(\mathbf{w})}\Big(|\mathbf{F}|_{\gamma}=\alpha_{\mathbf{w}}+dn\Big) \underset{n\to\infty}\sim  Z_{\mathbf{w}}\sqrt{\frac{\gamma\cdot\mathbf{a}}{2\pi d \sigma^2 n^3}},
\end{equation}
where $\mathbf{a}$ is the left eigenvector of the mean matrix $M$, and $\sigma^2$ was defined in Section \ref{gen1}. The actual values do not matter much however, the important part is that the right-hand side is $Z_{\mathbf{w}}$ divided by $n^{3/2}$, times a constant. We will prove this by using analytic methods, notably the smooth implicit-function schema theorem (see notably \cite{FS}, Section VII.4 and \cite{MeirMoon}).
\subsubsection{Proving \eqref{bywordprime} for one-letter words}
Let $i\in[K]$ and, for appropriate $z\in\C$, let
	\[\psi_i(z)= \mathbb{E}^{(i)}\big[z^{|\mathbf{T}|_{\gamma}}\big]=\sum_{n\in\Z_+} \pr^{(i)}\big(|\mathbf{T}|_{\gamma}=n\big)z^n.
\]
This power series has non-negative coefficients, and, since $\zeta$ is critical, its radius of convergence is $1$. This is because $\psi_i(1)=1$ (since $\psi^{(i)}$ is the generating function of a probability distribution) and $\psi'_i(1)=\infty$ (Lemma \ref{expinf}). We let $\mathbb{D}$ be the open unit disk.
The periodicity structure of Section \ref{sec:per} lets us rewrite $\psi_i$ in a more precise way: there exists another power series $\phi_i$ such that
	\[\forall z\in\mathbb{D}, \psi_i(z)=z^{\alpha_i}\phi_i(z^d),
\]
and all the coefficients of $\phi_i$, except for a finite amount, are strictly positive. Our aim is then to show that the coefficient of $z^n$ in $\phi_i$ behaves like $n^{-3/2}$ as $n$ tends to infinity.

\medskip

Recall from Section \ref{gen1} the distribution $\pr^{(i)}_{\text{cut}_i}$ of the Galton-Watson tree cut at its first generation of type $i$. Given such a cut tree $\mathbf{t}$, we call $p_{\mathbf{t}}$ its number of leaves of type $i$. We obtain from the Galton-Watson construction the following equation:
	\[\psi_i(z)=z^{\gamma_i}\sum_{\mathbf{t}}\pr^{(i)}_{\text{cut}_i}(\mathbf{t}) z^{\sum_{j\neq i}\gamma_j\#_{j}(\mathbf{t})} \big(\psi_i(z)\big)^{p_{\mathbf{t}}}.
\]
This can be refined with the periodicity structure: we know from Proposition \ref{per} that, if $\pr^{(i)}_{\text{cut}_i}(\mathbf{t})>0$, then $\alpha_i\equiv \gamma_i+\sum_{j\neq i}\gamma_j\#_{j}(\mathbf{t})+p_{\mathbf{t}}\alpha_i\pmod d$. We let $n_{\mathbf{t}}\in\Z_+$ be such that $\gamma_i+\sum_{j\neq i}\gamma_j\#_{j}(\mathbf{t})+p_{\mathbf{t}}\alpha_i=\alpha_i+n_{\mathbf{t}}d$, and then obtain
	\[z^\alpha_i\phi_i(z^d)=\sum_{\mathbf{t}}\pr^{(i)}_{\text{cut}_i}(\mathbf{t}) z^{\alpha_i+dn_{\mathbf{t}}}\phi_i(z^d),
\]
which reduces to
	\[\phi_i(z)=\sum_{\mathbf{t}}\pr^{(i)}_{\text{cut}_i}(\mathbf{t}) z^{n_{\mathbf{t}}}\big(\phi_i(z)\big)^{p_{\mathbf{t}}}.
\]

The function $\phi_i$ thus solves
	\[\phi_i(z) = G\big(z,\phi_i(z)\big)
\]
where, for appropriate $z$ and $w$,
	\[G(z,w)=\mathbb{E}^{(i)}_{\text{cut}_i}\big[ z^{n_{\mathbf{T}}}w^{p_{\mathbf{T}}}\big].
\]

We will now apply smooth implicit-function schema theorem, as stated in \cite{FS}, Theorem VII.3. We have to check several conditions on the double power series $G(z,w)=\sum_{n,m}g_{m,n}z^mw^n$ with positive coefficients first.
\begin{itemize}

\item We show that $G$ is analytic in a domain $\{|z|<R, |w|<R\}$ with $R>1$. Because of regular criticality and Proposition \ref{moments}, the number of vertices lying before or on the first generation of type $i$ are both exponentially integrable variables (in the sense of Appendix \ref{sec:EI}), thus their sum, which is the total number of vertices lying before the first generation of type $i$, is also exponentially integrable. Thus there exists $z>1$ such that $\mathbb{E}^{(i)}_{\text{cut}_i}\big[ z^{\#\T}\big]<\infty$, and then bounding $|\T|_{\gamma}$ by $\gamma_{\max} \#\T$ ($\gamma_{\max}$ being the highest value of $\gamma_i,$ $i\in[K]$) and rewriting $n_{\mathbf{T}}$ in terms of $|\T|_{\gamma}$, we get $R>1$ such that $G(R,R)<\infty$.
\item Unlike the assumptions of \cite{FS}, it is possible that $g_{0,0}=0$ (for example if $\gamma_i=0$ and an individual of type $i$ can die without giving birth to any offspring), but this is just an unneeded normalization assumption. We do know however that the coefficient for $g_{0,1}\neq1$ and that $g_{0,n}\neq0$ for some $n\geq 2$ since the measure $\mu_{i,i}$ has expected value $1$ and non-zero variance.
\item The pair $(r,s)=(1,1)$ lies inside the domain of analyticity of $G$ and satisfies the so-called \emph{characteristic system}
	\[G(r,s)=s \qquad\text{ and }\qquad \partial_{w}G(r,s)=1.
\]
Of course, in our setting, we are just saying that the coefficients of $G$ sum up to $1$ and that the average of $\mu_{i,i}$ is $1$, which we know since Proposition \ref{moments}.
\end{itemize}

Knowing all of this and the fact that $\phi_i$ is aperiodic (in the sense of \cite{FS}, since only a finite number of its coefficients are not $0$), the analytic implicit-function schema gives us the following estimate for the coefficient of $z^n$ in $\phi_i$:
	\[\pr^{(i)}\Big(|\mathbf{F}|_{\gamma}=\alpha_i+dn\Big) \underset{n\to\infty}\sim \sqrt{\frac{\partial_zG(1,1)}{2\pi\partial^2_{ww}G(1,1)n^3}}.
\]
Proposition \ref{moments} gives us the wanted values for the partial derivatives: 
\begin{align*}
\partial_zG(1,1)&=\mathbb{E}_{\text{cut}_i}\big[n_{\T}\big]\\
                &=\frac{1}{d}\mathbb{E}_{\text{cut}_i}\big[\gamma_i+(p_{\T}-1)\alpha_i+\sum_{j\neq i}\gamma_j\#_{j}(\mathbf{T})\big]\\
                &=\frac{1}{d}\Big(\gamma_i+0+\sum_{j\neq i}\frac{\gamma_j a_j}{a_i}\Big)\\
                &=\frac{\gamma\cdot \mathbf{a}}{da_i},
\end{align*}
and
\begin{align*}
\partial^{2}_{ww}G(1,1)&=\mathbb{E}_{\text{cut}_i}\big[p_{\T}(p_{\T}-1)\big]\\
                       &=\frac{\sigma^2}{a_ib_i^2},
\end{align*}
and this ends our proof. \qed

\subsubsection{Moving on to general words}
The general case of \eqref{bywordprime} follows from the following lemma:
\begin{lemma}
Let $a>1$ and let $X$ and $Y$ be two independent integer-valued random variables such that 
	\[\pr(X=n)\underset{n\to\infty}\sim \frac{C_X}{n^a} \qquad \text{ and } \qquad \pr(Y=n)\underset{n\to\infty}\sim \frac{C_Y}{n^a}.
\]

Then we also have 
	\[\pr(X+Y=n)\underset{n\to\infty}\sim \frac{C_X+C_Y}{n^a}
\]
\end{lemma}

\begin{proof} We will separately show that 
	\[ \underset{n\to\infty}\limsup\; n^a\pr(X+Y=n)\leq C_X+C_Y
\]	
and
	\[\underset{n\to\infty}\liminf\; n^a\pr(X+Y=n)\geq C_X+C_Y
\]

For $n\in\Z_+$, let $x_n=\pr(X=n)$, $y_n=\pr(Y=n)$ and $z_n=\pr(X+Y=n)=\sum_{k=0}^nx_ky_{n-k}$. Cut the sum the following way:

\begin{equation}\label{eq:cut}z_n= \sum_{k=0}^{K} x_ky_{n-k} + \sum_{k=K+1}^{n-K-1} x_ky_{n-k} + \sum_{k=n-K}^{n} x_ky_{n-k}.
\end{equation}

For the lower bound, let $\epsilon>0$, and choose $K$ large enough that $\sum_{k=0}^K x_k\geq (1-\epsilon)$, $\sum_{k=0}^K y_k\geq (1-\epsilon)$ and, for $n$ larger than $K$, $x_n\geq(1-\epsilon)C_X n^{-a}$ and $y_n\geq(1-\epsilon)C_Y n^{-a}$. Now take $n\geq 2K$. In the first sum, use $y_{n-k}\geq (1-\epsilon)C_Y(n-K)^{-a}$, and in the third, use $x_k\geq (1-\epsilon)C_X(n-K)^{-a}$ to obtain\[ z_n \geq (1-\epsilon)(n-K)^{-a} \big(C_X\sum_{k=0}^n y_k+C_Y\sum_{k=0}^n x_k\big)\geq (1-\epsilon)^2(n-K)^{-a}(C_X+C_Y).\]
Taking $n$ to infinity, we get 
	\[\underset{n\to\infty}\liminf\; n^a z_n\geq (1-\epsilon)^2(C_X+C_Y),
\]
and letting $\epsilon$ tend to $0$ gives us the lower bound.

The upper bound will require more work. Let $\epsilon>0$ and $0<\epsilon<1/2$, we will do the same cut as in Equation (\ref{eq:cut}), but with a varying $K$, equal to $\lfloor \epsilon n\rfloor.$ Take $n$ large enough such that, for $k\geq \lfloor\epsilon n\rfloor$, $k^{a}x_k\leq (1+\epsilon)C_X$ and $k^{a}y_k\leq (1+\epsilon)C_Y.$ Write in the first sum $y_{n-k}\leq (1+\epsilon)C_Y(n-\lfloor \epsilon n\rfloor)^{-a}$  and in the third one $x_k\leq (1+\epsilon)C_X(n-\lfloor \epsilon n\rfloor)^{-a}$, while for the middle one we use $x_ky_{n-k}\leq (1+\epsilon)^2C_XC_Y \lfloor \epsilon n\rfloor^{-2a}$. We then have
\begin{align*}
z_n &\leq \sum_{k=0}^{\lfloor \epsilon n\rfloor}x_k(1+\epsilon)(n-\lfloor \epsilon n\rfloor)^{-a} C_Y+\sum_{k=0}^{\lfloor \epsilon n\rfloor}y_k(1+\epsilon)(n-\lfloor \epsilon n\rfloor)^{-a} C_X + \sum_{k=0}^n (1+\epsilon)^2C_XC_Y\lfloor \epsilon n\rfloor^{-2a} \\ &\leq(1+\epsilon)(C_X+C_Y)(n-\lfloor \epsilon n\rfloor)^{-a}+(1+\epsilon)^2C_XC_Y(\lfloor \epsilon n\rfloor) n\lfloor \epsilon n\rfloor^{-2a}.
\end{align*}
Since $a>1,$ we have $1-2a<a,$ and thus the last term is negligible compared to $n^{-a}$. Hence $\limsup n^az_n\leq (1+\epsilon)(1-\epsilon)^{-a},$ and letting $\epsilon$ tend to $0$ gives us the wanted bound.
\end{proof}

The case $a=3/2$, coupled with a simple induction then proves \eqref{bywordprime} for general $\mathbf{w}\in\W_K$.
\bigskip
\section{Background on random planar maps}
\subsection{Planar maps}
As stated in the introduction, a planar map is a proper embedding $m$ of a finite connected planar graph in the sphere, in the sense that edges do not intersect. These are taken up to orientation-preserving homeomorphisms of the sphere, thus making them combinatorial objects. We call \emph{faces} of a map $m$ the connected components of its complement in the sphere, and let $\mathcal{F}_m$ be their set. The \emph{degree} of a face $f$, denoted by $\deg(f)$, is the number of edges it is adjacent to, counting multiplicity: we count every edge as many times as we encounter it when circling around $f$. The numbers of vertices, edges and faces of a map are respectively denoted by $\#V(m)$, $\#E(m)$ and $\#F(m)$. Finally, the graph distance on $m$ is denoted by $d$.

We are going to look at maps which are both \emph{rooted} and \emph{pointed}. These are triplets $(m,e,r)$, where $m$ is a planar map, $e$ is an oriented edge of $m$ called the root edge, starting at a vertex $e^-$ and pointing to a vertex $e^+$, and $r$ is a vertex of $m$. We call $\M$ the set of all such maps and $\M_n$ the set of such maps with $n$ vertices for $n\in\N$. A map $(m,e,r)$ will be called positive (resp. null, negative) if $d(r,e^+)=d(r,e^-)+1$ (resp. $d(r,e^-)$, $d(r,e^-)-1$). We call $\M^+$, $\M^0$ and $\M^-$ the corresponding sets of maps and, for $n\in\N$, $\M_n^+$, $\M_n^0$ and $M_n^-$ the corresponding sets of maps which have $n$ vertices. Since there is a trivial bijection between positive and negative maps, we will mostly restrict ourselves to $\M^+$ and $\M^0$. By convention, we add to $\M^+$ the vertex map $\dagger$, which consists of one vertex, no edges and one face.

\subsection{Boltzmann distributions}\label{chap4:defboltzmann}
Let $\mathbf{q}=(q_n)_{n\in\N}$ be a sequence of nonnegative numbers such that there exists $i\geq 3$ with $q_i>0$. For any map $m$, let 
	\[W_{\mathbf{q}}(m)=\prod_{f\in \mathcal{F}_m} q_{\deg(f)}.
\]
Note that this quantity only depends on the map $m$, and not on any root $r$ or point $m$.
We say that the sequence $\mathbf{q}$ is \emph{admissible} if the sum
	\[Z_{\mathbf{q}}= \sum_{(m,e,r)\in\M} W_{\mathbf{q}}(m)
\]
is finite. When $\mathbf{q}$ is admissible, we can define the Boltzmann probability distribution $B_{\mathbf{q}}$ by setting, for a pointed rooted map $(m,e,r)$,
	\[B_{\mathbf{q}}(m,e,r)=\frac{W_{\mathbf{q}}(m)}{Z_{\mathbf{q}}}.
\] We also introduce the versions of $B_{\mathbf{q}}$ conditioned to be positive or null: let $Z_{\mathbf{q}}^+= \sum_{(m,e,r)\in\M^+} W_{\mathbf{q}}(m)$ and $Z_{\mathbf{q}}^0= \sum_{(m,e,r)\in\M^0} W_{\mathbf{q}}(m)$ and, for any map $(m,e,r)$, $B^+_{\mathbf{q}}(m,e,r)=\frac{W_{\mathbf{q}}(m)}{Z_{\mathbf{q}}^+}$ if it is positive and $B^0_{\mathbf{q}}(m,e,r)=\frac{W_{\mathbf{q}}(m)}{Z_{\mathbf{q}}^0}$ if it is null. 

For nonnegative numbers $x$ and $y$, let
	\[f^{\bul}(x,y)=\sum_{k,k'}{\binom{2k+k'+1}{k+1}} {\binom{k+k'}{k}}q_{2+2k+k'}\,x^ky^{k'}
\]
and
	\[f^{\dia}(x,y)=\sum_{k,k'}{\binom{2k+k'}{k}} {\binom{k+k'}{k}}q_{1+2k+k'}\,x^ky^{k'}.
\]
It was shown in \cite{M06}, Proposition 1, that $\mathbf{q}$ is admissible if and only if the system
\begin{align}
1-\frac{1}{x}=f^{\bul}(x,y) \\
y=f^{\dia}(x,y)
\end{align}
has a solution with $x>1$, such that the spectral radius of the matrix

\[  \left( \begin{array}{ccc}
0         &           0            &          x-1 \\
\frac{x}{y}\partial_xf^{\dia}(x,y) & \partial_yf^{\dia}(x,y)    & 0 \\
\frac{x^2}{x-1}\partial_xf^{\bul}(x,y) & \frac{xy}{x-1}\partial_yf^{\bul}(x,y) & 0
\end{array} \right)
\]
is smaller than or equal to $1$. The existence of such a solution implies its uniqueness, with $x=Z_{\mathbf{q}}^+$ and $y=\sqrt{Z_{\mathbf{q}}^0}$. We let $Z^{\dia}_{\mathbf{q}}=\sqrt{Z_{\mathbf{q}}^0}$.

We then say that $\mathbf{q}$ is \emph{critical} if the spectral radius of the aforementioned matrix is exactly $1$.
and that it is \emph{regular critical} if, moreover, for some $\epsilon>0$, we have $f^{\bul}(Z_{\mathbf{q}}^++\epsilon,Z_{\mathbf{q}^{\dia}}+\epsilon)<\infty$.

\bigskip

\noindent\textbf{Random non-pointed maps.} We will also occasionally consider rooted maps $(m,e)$ without any specified point $r$. If $\mathbf{q}$ is admissible, we let $B^{\emptyset}_{\mathbf{q}}$ be the probability measure on the set of rooted maps such that, for a rooted map $(m,e)$,
	\[B^{\emptyset}_{\mathbf{q}}(m,e)=\frac{W_{\mathbf{q}}(m)}{Z^{\emptyset}_{\mathbf{q}}}
\]
where $Z^{\emptyset}_{\mathbf{q}}$ is an appropriate constant.

Note that, if a random rooted and pointed map $(M,E,R)$ has distribution $B_{\mathbf{q}}$, then the distribution of $(M,E)$ (ignoring $R$) is not $B^{\emptyset}_{\mathbf{q}}$, but $B^{\emptyset}_{\mathbf{q}}$ biased by the number of vertices: if $(m,e)$ is a rooted map with $n$ vertices, then $\sum_{r\in m} B_{\mathbf{q}}(m,e,r)$ is proportional to $nB^{\emptyset}_{\mathbf{q}}$. This is because, there are exactly $n$ ways of pointing $(m,e)$, and they all lead to a different rooted and pointed map.
\subsection{The Bouttier-Di Francesco-Guitter bijection}\label{sec:BDFG}
In \cite{BDFG} was exposed a bijection between rooted and pointed maps and a certain class of $4$-type labelled trees called \emph{mobiles}. Let us quickly recall the facts here, with a few variations to make the bijection more adapted to our study.
\subsubsection{Mobiles}
A finite spatial $4$-type tree $(\mathbf{t},\mathbf{e},\mathbf{l})$ is called a \emph{mobile} if the types satisfy the following conditions:
\begin{itemize}
\item The root has type $1$ or $2$,
\item The children of a vertex of type $1$ all have type $3$,
\item If a vertex has type $2$, then it has only one child, which has type $4$, except if it is the root, if $\emptyset$ has type $2$ then it has exactly two children, both of type $4$,
\item Vertices of type $3$ and $4$ can only have children of types $1$ and $2$,
\end{itemize}
and the labels satisfy the following conditions:
\begin{itemize}
\item Vertices of type $1$ and $3$ have integer labels, vertices of type $2$ and $4$ have labels in $\Z+\frac{1}{2}$,
\item The root has label $0$ if it is of type $1$, $\frac{1}{2}$ if it is of type $2$,
\item Vertices of type $3$ or $4$ have the same label as their father.
\item If $u\in \mathbf{t}$ has type $3$ or $4$, let by convention $u0=u\underline{k_u(\mathbf{t})+1}=u^-$. Then, for all $i\in\{0,\ldots,k_u(\mathbf{t})\}$, $\mathbf{l}\big(u\underline{i+1}\big)-\mathbf{l}(ui)\geq -\frac{1}{2} (\mathbbm{1}_{\{\mathbf{e}(ui)=1\}} + \mathbbm{1}_{\{\mathbf{e}(u\underline{i+1})=1\}})$.
\end{itemize}

The notation $u\underline{i+1}$ means that we are looking at $i+1$ as a letter, the word $u\underline{i+1}$ being the concatenation of $u$ and $i+1$.

Traditionally, vertices of type $1$ are represented as white circles $\bigcirc$, vertices of type $2$ are ``flags" $\dia$ while the other two types are dots $\bul$. Notice also that we do not need to mention the labels of vertices with type $3$ and $4$ since the label of such a vertex is the same as that of its father. We let $\mathbb{T}_M$ be the set of finite mobiles, $\mathbb{T}_M^+$ be the set of finite mobiles such that $\mathbf{e}(\emptyset)=1$ and $\mathbb{T}_M^0$ be the set of finite mobiles such that $\mathbf{e}(\emptyset)=2$.

\vspace{2.6cm}

\begin{figure}[ht]
\centering
\begingroup%
  \makeatletter%
  \providecommand\color[2][]{%
    \errmessage{(Inkscape) Color is used for the text in Inkscape, but the package 'color.sty' is not loaded}%
    \renewcommand\color[2][]{}%
  }%
  \providecommand\transparent[1]{%
    \errmessage{(Inkscape) Transparency is used (non-zero) for the text in Inkscape, but the package 'transparent.sty' is not loaded}%
    \renewcommand\transparent[1]{}%
  }%
  \providecommand\rotatebox[2]{#2}%
  \ifx\svgwidth\undefined%
    \setlength{\unitlength}{278.96328125bp}%
    \ifx\svgscale\undefined%
      \relax%
    \else%
      \setlength{\unitlength}{\unitlength * \real{\svgscale}}%
    \fi%
  \else%
    \setlength{\unitlength}{\svgwidth}%
  \fi%
  \global\let\svgwidth\undefined%
  \global\let\svgscale\undefined%
  \makeatother%
  \begin{picture}(1,0.65885037)%
    \put(0,0){\includegraphics[width=\unitlength]{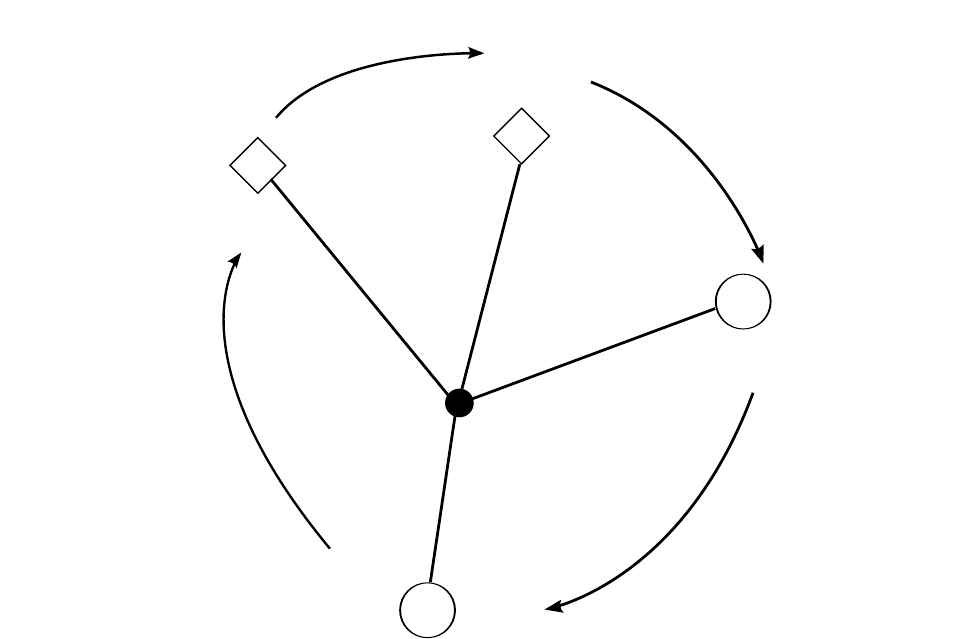}}%
    \put(0.28176451,0.64539934){\color[rgb]{0,0,0}\makebox(0,0)[lb]{\smash{$0,1,2,3,\ldots$}}}%
    \put(0.69361727,0.60528351){\color[rgb]{0,0,0}\makebox(0,0)[lb]{\smash{$-\frac{1}{2},\frac{1}{2},\frac{3}{2},\ldots$}}}%
    \put(0.74646028,0.04590319){\color[rgb]{0,0,0}\makebox(0,0)[lb]{\smash{$-1,0,1,2,\ldots$}}}%
    \put(-0.00061332,0.17650393){\color[rgb]{0,0,0}\makebox(0,0)[lb]{\smash{$-\frac{1}{2},\frac{1}{2},\frac{3}{2},\ldots$}}}%
  \end{picture}%
\endgroup%

\caption{The authorized labelling differences when circling around a vertex of type $3$ or $4$.}
\end{figure}
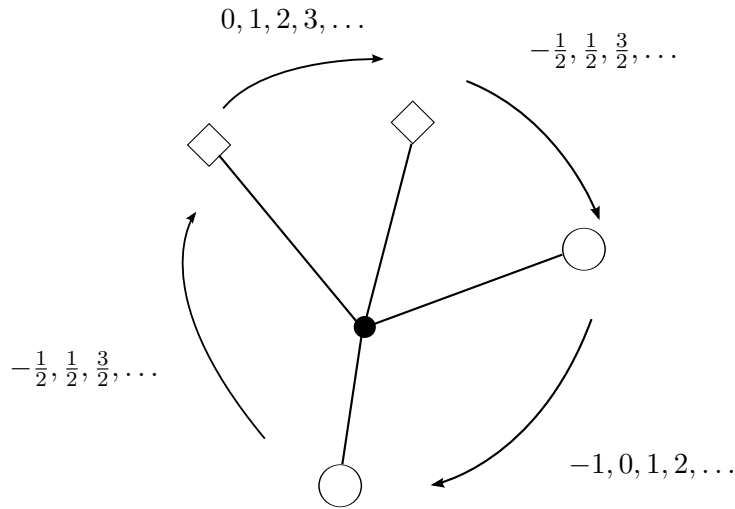

\newpage

\begin{figure}[!ht]
\centering
\begingroup%
  \makeatletter%
  \providecommand\color[2][]{%
    \errmessage{(Inkscape) Color is used for the text in Inkscape, but the package 'color.sty' is not loaded}%
    \renewcommand\color[2][]{}%
  }%
  \providecommand\transparent[1]{%
    \errmessage{(Inkscape) Transparency is used (non-zero) for the text in Inkscape, but the package 'transparent.sty' is not loaded}%
    \renewcommand\transparent[1]{}%
  }%
  \providecommand\rotatebox[2]{#2}%
  \ifx\svgwidth\undefined%
    \setlength{\unitlength}{179.74675767bp}%
    \ifx\svgscale\undefined%
      \relax%
    \else%
      \setlength{\unitlength}{\unitlength * \real{\svgscale}}%
    \fi%
  \else%
    \setlength{\unitlength}{\svgwidth}%
  \fi%
  \global\let\svgwidth\undefined%
  \global\let\svgscale\undefined%
  \makeatother%
  \begin{picture}(1,1.25632968)%
    \put(0,0){\includegraphics[width=\unitlength,page=1]{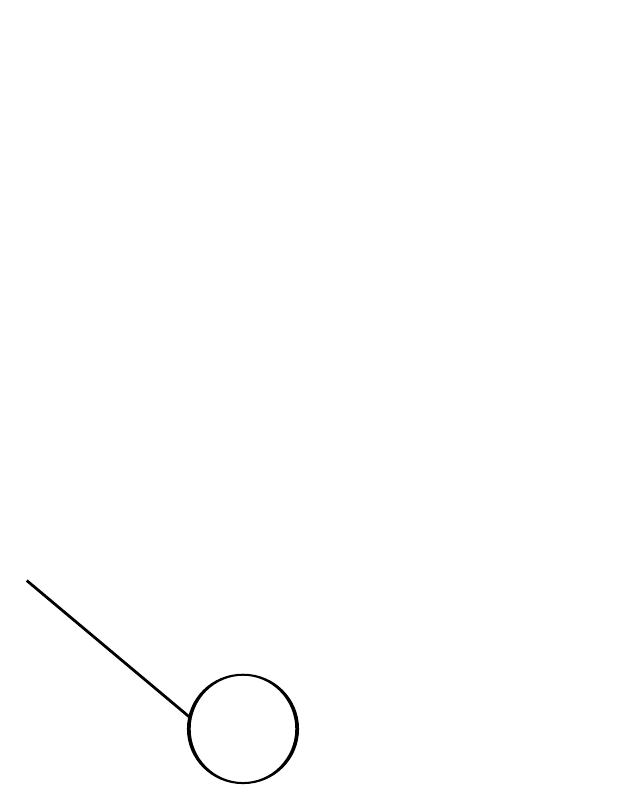}}%
    \put(0.37330084,0.06815543){\color[rgb]{0,0,0}\makebox(0,0)[lb]{\smash{$0$}}}%
    \put(0,0){\includegraphics[width=\unitlength,page=2]{mobilevrai.pdf}}%
    \put(0.07893664,0.62509988){\color[rgb]{0,0,0}\makebox(0,0)[lb]{\smash{$-1$}}}%
    \put(0,0){\includegraphics[width=\unitlength,page=3]{mobilevrai.pdf}}%
    \put(0.33291375,0.63864555){\color[rgb]{0,0,0}\makebox(0,0)[lb]{\smash{$-\frac{1}{2}$}}}%
    \put(0,0){\includegraphics[width=\unitlength,page=4]{mobilevrai.pdf}}%
    \put(0.33993114,1.14948958){\color[rgb]{0,0,0}\makebox(0,0)[lb]{\smash{$-1$}}}%
    \put(0,0){\includegraphics[width=\unitlength,page=5]{mobilevrai.pdf}}%
    \put(0.63268609,0.63545734){\color[rgb]{0,0,0}\makebox(0,0)[lb]{\smash{$\frac{1}{2}$}}}%
    \put(0,0){\includegraphics[width=\unitlength,page=6]{mobilevrai.pdf}}%
    \put(0.63705109,1.14956307){\color[rgb]{0,0,0}\makebox(0,0)[lb]{\smash{$1$}}}%
    \put(0,0){\includegraphics[width=\unitlength,page=7]{mobilevrai.pdf}}%
    \put(0.89062331,1.14549267){\color[rgb]{0,0,0}\makebox(0,0)[lb]{\smash{$1$}}}%
  \end{picture}%
\endgroup%
\caption{An example of a mobile, with root of type $1$.}
\label{figmobile}
\end{figure}

\subsubsection{The bijection}\label{sec:thebijection}
Let $(\mathbf{t},\mathbf{e},\mathbf{l})$ be a mobile and let us describe how to transform it into a map. Let $v_1,v_2,\ldots,v_p$ be, in order, the vertices of type $1$ or $2$ of $\mathbf{t}$ appearing in the standard contour process and $e_1,e_2,\ldots,e_p$ and $l_1,l_2,\ldots,l_p$ be the corresponding types and labels. We refer to $v_1,\ldots,v_p$ as the \emph{corners} of the tree because a vertex will be visited a number of times equal to the number of angular sectors around it delimited by the tree. Draw $\mathbf{t}$ in the plane and add an extra type $1$ vertex $r$ outside of $\mathbf{t}$, giving it label $\underset{\mathbf{e}(u)=1}\min \mathbf{l}(u)-1$. Now, for every $i\in [p]$, define the successor of the $i$-th corner as the next corner of type $1$ with label $l_i-1$ if $e_i=1$ and $l_i-\frac{1}{2}$ if $e_i=2$. If there is no such vertex, then let its successor be $r$. In both cases, draw an arc between $v_i$ and the successor. This construction can be done without having any of the arcs intersect. Now erase all the original edges of the tree, as well as vertices of types $3$ and $4$. Erase as well all the vertices of type $2$, merging the corresponding pairs of arcs. We are left with a planar map, with a distinguished vertex $r$. The root edge depends on the type of the root of the tree: if $\mathbf{e}(\emptyset)=1$ then we let the root edge be the first arc which was drawn (have it point to $\emptyset$ for a positive map, and away from $\emptyset$ for a negative map). If $\mathbf{e}(\emptyset)=2$ then we let the root edge be the result of the merging of the two edges adjacent to $\emptyset$, pointing to the successor of the first corner encountered in the contour process.

This construction gives us two bijections: one between $\mathbb{T}_M^+$ and $\mathcal{M}^+$ and one between $\mathbb{T}_M^0$ and $\mathcal{M}^0$, which we both call $\Psi$.

\newpage

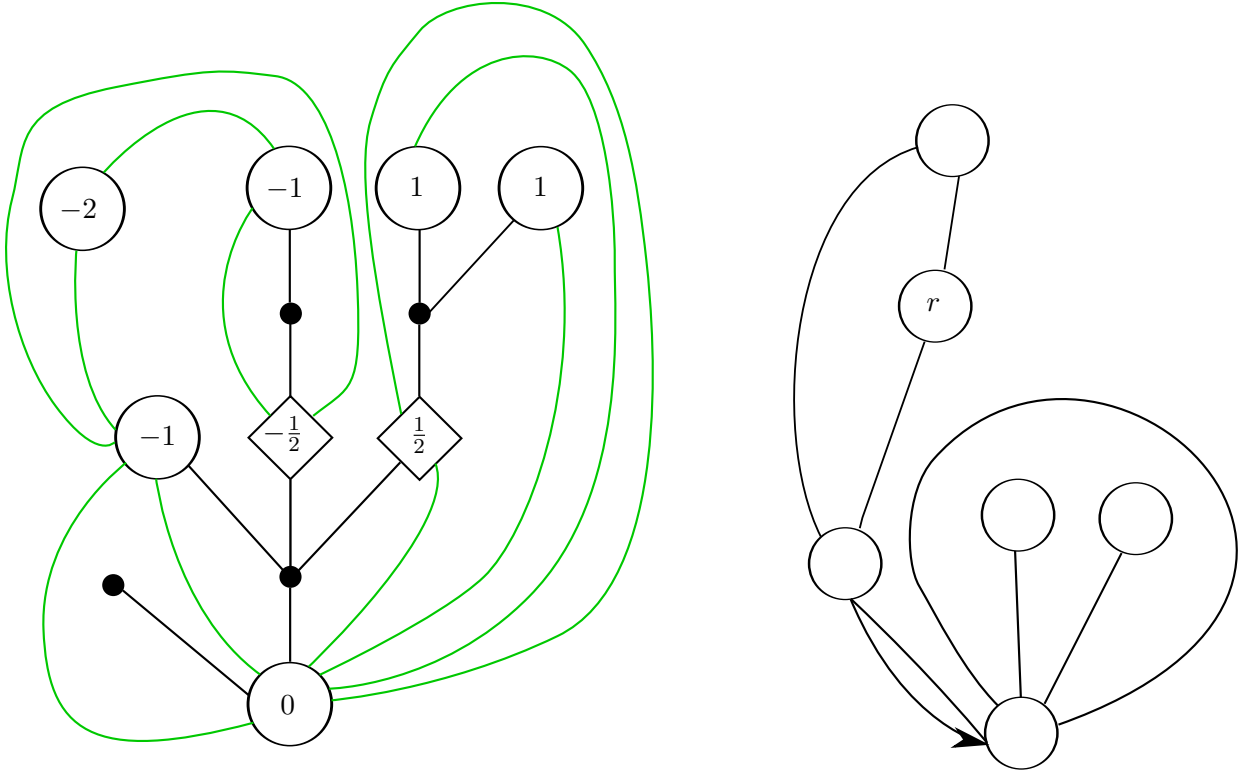
\begin{figure}[!htpb]
\centering
\begingroup%
  \makeatletter%
  \providecommand\color[2][]{%
    \errmessage{(Inkscape) Color is used for the text in Inkscape, but the package 'color.sty' is not loaded}%
    \renewcommand\color[2][]{}%
  }%
  \providecommand\transparent[1]{%
    \errmessage{(Inkscape) Transparency is used (non-zero) for the text in Inkscape, but the package 'transparent.sty' is not loaded}%
    \renewcommand\transparent[1]{}%
  }%
  \providecommand\rotatebox[2]{#2}%
  \ifx\svgwidth\undefined%
    \setlength{\unitlength}{459.80242921bp}%
    \ifx\svgscale\undefined%
      \relax%
    \else%
      \setlength{\unitlength}{\unitlength * \real{\svgscale}}%
    \fi%
  \else%
    \setlength{\unitlength}{\svgwidth}%
  \fi%
  \global\let\svgwidth\undefined%
  \global\let\svgscale\undefined%
  \makeatother%
  \begin{picture}(1,0.62752515)%
    \put(0,0){\includegraphics[width=\unitlength]{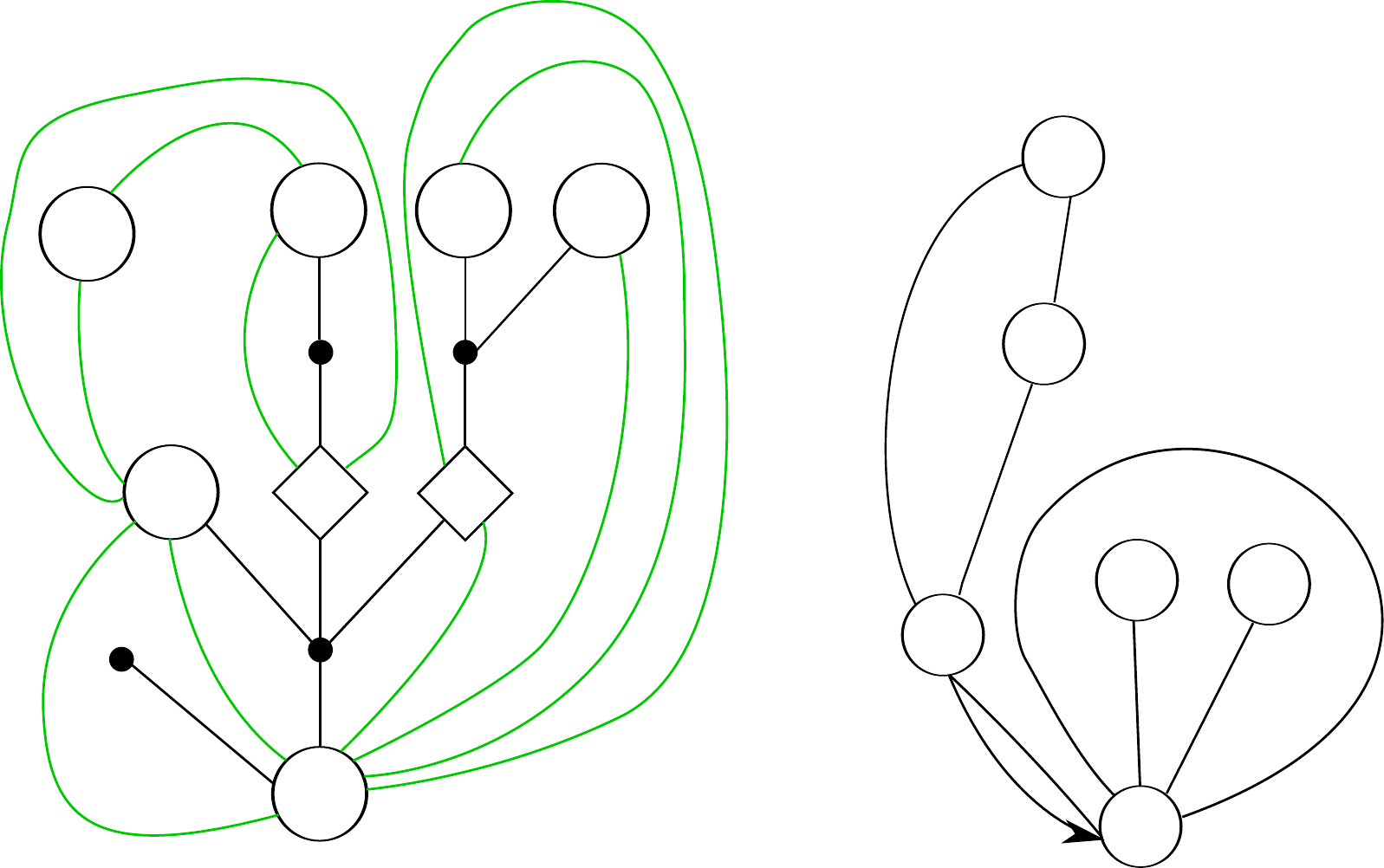}}%
    \put(0.22341962,0.04562033){\color[rgb]{0,0,0}\makebox(0,0)[lb]{\smash{$0$}}}%
    \put(0.1085877,0.26522478){\color[rgb]{0,0,0}\makebox(0,0)[lb]{\smash{$-1$}}}%
    \put(0.20901104,0.26863732){\color[rgb]{0,0,0}\makebox(0,0)[lb]{\smash{$-\frac{1}{2}$}}}%
    \put(0.21175429,0.46833735){\color[rgb]{0,0,0}\makebox(0,0)[lb]{\smash{$-1$}}}%
    \put(0.32888035,0.26739103){\color[rgb]{0,0,0}\makebox(0,0)[lb]{\smash{$\frac{1}{2}$}}}%
    \put(0.32816028,0.46879047){\color[rgb]{0,0,0}\makebox(0,0)[lb]{\smash{$1$}}}%
    \put(0.42855119,0.46879047){\color[rgb]{0,0,0}\makebox(0,0)[lb]{\smash{$1$}}}%
    \put(0.74737151,0.3746592){\color[rgb]{0,0,0}\makebox(0,0)[lb]{\smash{$r$}}}%
    \put(0.04432822,0.4512151){\color[rgb]{0,0,0}\makebox(0,0)[lb]{\smash{$-2$}}}%
  \end{picture}%
\endgroup%
\caption{Having added a vertex with label $-2$ to the mobile of Figure \ref{figmobile}, we transform it into a map.}
\end{figure}

It was shown in \cite{M06} that the BDFG bijection serves as a link between Galton-Watson mobiles and Boltzmann maps.

\begin{prop}\label{bij} Consider and admissible weight sequence $\mathbf{q}$ and define an unordered $4$-type offspring distribution $\mathbf{\mu}$ by
\begin{align*}
&\mu^{(1)}(0,0,k,0)=\frac{1}{Z^+_\mathbf{q}}(1-\frac{1}{Z^+_\mathbf{q}})^k \\
	&\mu^{(2)}(0,0,0,1)=1 \\
	&\mu^{(3)}(k,k',0,0)=\frac{(Z^+_\mathbf{q})^k (Z^{\dia}_\mathbf{q})^{k'}{\binom{2k+k'+1}{k+1}} {\binom{k+k'}{k}}q_{2+2k+k'}}{f^{\bul}(Z^+_\mathbf{q},Z^{\dia}_\mathbf{q})} \\
	&\mu^{(4)}(k,k',0,0)=\frac{(Z^+_\mathbf{q})^k (Z^{\dia}_\mathbf{q})^{k'}{\binom{2k+k'}{k}}{\binom{k+k'}{k}}q_{1+2k+k'}}{f^{\dia}(Z^+_\mathbf{q},Z^{\dia}_\mathbf{q})}.
\end{align*}

Let then $\z$ be the ordered offspring distribution which is uniform ordering of $\mu$, as explained in Section \ref{chap4:basicdef}. This offspring distribution is irreducible, and it is critical (resp. regular critical) if the weight sequence $\mathbf{q}$ is critical (resp. regular critical), while it is subcritical if $\mathbf{q}$ is admissible but not critical. Define also, for all ordered offspring type-list $\mathbf{w}$, $\nu^{(i)}_{\mathbf{w}}$ as the uniform measure on the set $D^{(i)}_{\mathbf{w}}$ of allowed displacements to have a mobile, which is precisely $D^{(i)}_{\mathbf{w}} = \{0\}^{|\mathbf{w}|}$ if $i=1$ or $i=2$ and
	\[ D^{(i)}_{\mathbf{w}} = 	\big\{\mathbf{y}=(y_i)_{i\in[|\mathbf{w}|]}:\; \ \forall i\in\{0,1,\ldots,|\mathbf{w}|\}, \ y_{i+1}-y_i + \frac{1}{2}(\mathbbm{1}_{\{w_i=1\}}+\mathbbm{1}_{\{w_{i+1}=1\}})\in\mathbf{Z_+} \big\},
\]
if $i=3$ or $i=4$, in which case we set by convention $w_0=w_{|\mathbf{w}|+1}=i-2$ and $y_0=y_{|\mathbf{w}|+1}=0$.

Then:
\begin{itemize}
\item if $(\mathbf{T,E,L})$ has distribution $\pr^{(1),(0)}_{\z,\nu}$, then the random map $\Psi(\mathbf{T,E,L})$ has distribution $\mathbb{B}^+_{\mathbf{q}}$.
\item if $(\mathbf{F,E,L})$ is a forest with distribution $\pr^{(2,2),(\frac{1}{2},\frac{1}{2})}_{\z,\nu}$, consider the mobile formed by merging both tree components at their roots. The image of this mobile by $\Psi$ has law $\mathbb{B}_{\mathbf{q}}^0$.
\end{itemize}

\end{prop}

\begin{rem} The operation of merging two trees at their roots can be formalized the following way. Consider two trees $(\mathbf{t}_1,\mathbf{e}_1)$ $(\mathbf{t}_2,\mathbf{e}_2)$ which are such that, in both trees, the root has type $2$ and has a unique child, with type $4$. For $u\in\mathbf{t}_2\setminus\{\emptyset\}$, we can write $u=1u^2\ldots u^k$. Let then $u'=2u^2\ldots u^k$, and let $\mathbf{t}_2'=\big\{u',\; u\in\mathbf{t}_2\setminus\{\emptyset\}\big\}$. We can now define $\mathbf{t}=\mathbf{t}_1\cup\mathbf{t}_2'$, which is easily checked to be a tree. Types can then simply be assigned by setting, for $u\in\mathbf{t}_1$, $\mathbf{e}(u)=\mathbf{e}_1(u)$ and, for $u\in\mathbf{t}_2\setminus\{\emptyset\}$, $\mathbf{e}(u')=\mathbf{e}_2(u)$.

This operation is of course continuous for the local convergence topology since, for any $k\in\Z_+$, the $k$-th generation of $\mathbf{t}$ is completely determined by the $k$-th generations of $\mathbf{t}_1$ and $\mathbf{t}_2$.
\end{rem}

\begin{rem} If the weight sequence $\mathbf{q}$ is such that $q_{2n+1}=0$ for all $n\in\Z_+$, then a $\mathbf{q}$-Boltzmann map is a.s. bipartite, which implies that there will be no vertices of type $2$ or $4$ in the corresponding tree, in which case we consider the mobile as a tree with two types, and it stays irreducible. Moreover, we then have $Z^{\dia}_{\mathbf{q}}=0$.
\end{rem}

\begin{rem} The number of vertices, edges and faces of the map can be read on the tree.
\begin{itemize}
\item $\#V\big(\Psi(\mathbf{T,E,L})\big)=1 + \#_1\mathbf{T}.$
\item $\#E\big(\Psi(\mathbf{T,E,L})\big)=|\mathbf{T}|_{\gamma}-1$ with $\gamma=(1,0,1,1).$
\item $\#F\big(\Psi(\mathbf{T,E,L})\big) =|\mathbf{T}|_{\gamma}$ with $\gamma=(0,0,1,1).$
\end{itemize}
\end{rem}
\subsection{Infinite maps and local convergence}
If $(m,e)$ is a rooted map and $k\in\N$, we let $B_{m,e}(k)$ be the map formed by all vertices whose graph distance to $e^+$ is less than or equal to $k$, and all edges connecting such vertices, except if the distance between each vertex of such an edge and $e^+$ is exactly $k$. The map $B_{m,e}(k)$ is still rooted at the same oriented edge $e$. For two rooted maps $(m,e)$ and $(m',e')$, let $d\big((m,e),(m',e')\big)=\frac{1}{1+p}$ where $p$ is the supremum of all integers $k$ such that $B_{m,e}(k)$ is equivalent to $B_{m',e'}(k)$. This defines a metric on the set of rooted maps. Call then $\overline{\mathcal{M}}$ the completion of this set. Elements of $\overline{\mathcal{M}}$ which are not finite maps are then called \emph{infinite maps}, which we mostly consider as a sequence of compatible finite maps: $(m,e)=(m_i,e_i)_{i\in\N}$ with $(m_i,e_i)=B_{m_{i+1},e_{i+1}}(i)$ for all $i$. Note in particular that infinite maps are not pointed.


As with trees and forests, convergence in distribution is simply characterized: if $(M_n,E_n)_{n\in\N}$ is a sequence of random rooted maps, one can check that it converges in distribution to a certain random map $(M,E)$ if and only if, for all finite deterministic maps $(m',e')$ and all $k\in \N$, $\pr\big(B_{(M_n,E_n)}(k)=(m',e')\big)$ converges to $\pr\big(B_{(M,E)}(k)=(m',e')\big).$

\section{Convergence to infinite Boltzmann maps}\label{sec:cvcartes}
We now take a critical weight sequence $\mathbf{q}$, and take $\mu$, $\z$ and $\nu$ as defined in Proposition \ref{bij}. Since, in the BDFG bijection, the number of vertices, edges and faces of the map correspond to the vertices of certain types of the tree, we expect Theorem \ref{cvforests} to tell us that Boltzmann maps with large amounts of vertices converge locally. This section is dedicated to establishing the fact that this is indeed the case.

We first define three different periodicity factors $d_V$, $d_E$ and $d_F$ corresponding to vertices, edges and faces
\begin{align*}d_V&=gcd\Big(\big\{n\in\N:\,q_{2n+2}>0\big\}\cup\big\{m\in 2\Z_+ +1:\,q_{m+2}>0\big\}\Big) \\
	d_E&=gcd\Big(\big\{n\in\N:\,q_{2n}>0\big\}\cup\big\{m\in 2\Z_+ +1:\,q_m>0\big\}\Big)\\
	d_F&=	\begin{cases}
	1\text{ if }\exists n\in\N:\,q_{2n}>0 \\
	2\text{ otherwise.}
	\end{cases}
\end{align*}
Let also $\alpha_V=2$ and $\alpha_E=\alpha_F=0$.
\begin{theo}\label{chap4:thcvcartes}
Let $I\in\{V,E,F\}$. If $I=E$ or $I=F$, we also assume that $\mathbf{q}$ is regular critical.
For appropriate $n\in\N$, let $(M_n,E_n,R_n)$ be a variable with distribution $B_{\mathbf{q}}$, conditioned on $\#I(M)=n$. We then have 
	\[(M_n,E_n) \underset{\underset{n\in \alpha_I+d_I\Z_+}{n\to\infty}}\Longrightarrow (M_{\infty},E_{\infty})
\]
in distribution for the local convergence, where $(M_{\infty},E_{\infty})$ is an infinite rooted map which we call the infinite $\mathbf{q}$-Boltzmann map.

If $I=E$ or $I=F$, assuming $\mathbf{q}$ is regular critical, consider now $(M_n,E_n)$ with distribution $B^{\emptyset}_{\mathbf{q}}$ conditioned on $\#I(M)=n$. We then also have
	\[(M_n,E_n) \underset{\underset{n\in \alpha_I+d_I\Z_+}{n\to\infty}}\Longrightarrow (M_{\infty},E_{\infty})
\]
with the same limiting map $(M_{\infty},E_{\infty})$.
\end{theo}

The choice of the subsequence $(\alpha_I+d_In)_{n\in\Z_+}$ is explained by the fact that, just as with trees, the number of vertices/edges/faces of a map with distribution $B_{\mathbf{q}}$ can only be of the form $\alpha_I+d_In$ for integer $n$, and this has non-zero probability for $n$ large enough. This will be explained in Section \ref{sec:mapperiod}.

The infinite map $(M_{\infty},E_{\infty})$ is moreover \emph{planar}, in the sense that it is possible to embed it in the plane in such a way that bounded subsets of the plane only encounter a finite number of edges, see Lemma \ref{lem:planar}.

\subsection{Two applications}

\subsubsection{The example of uniform $p$-angulations}
Here we take an integer $p\geq 3$ and consider maps which only have faces of degree $p$, which we call $p$-angulations. The well-known Euler's formula will show us that the number of vertices and edges of such a map are determined by its number of faces. Let $m$ be a finite $p$-angulation, and let $V$ be its number of vertices, $E$ be its number of edges and $F$ be its number of faces. Since each edge is adjacent to two faces, we have $p\#F(m)=2\#E(m)$. Euler's formula, on the other hand, states that $\#V(m)-\#E(m)+\#F(m)=2$. Combining the two shows that 
	\[\#V(m)=2+(\frac{p}{2}-1)\#F(m).
\]
Note that these relations imply that there is no real difference between pointed and non-pointed maps when looking at $p$-angulations, since a uniform pointed map with a fixed number of faces can be obtain by taking a uniform non-pointed map and uniformly choosing the specific point afterwards.

At this point, we must split the discussion according to the parity of $p$.

\medskip

\noindent\textit{Uniform infinite $2p$-angulation} Let $p\geq2$. It has been shown in \cite{MM07} that the weight sequence $\mathbf{q}$ defined by
	\[q_n=\frac{(p-1)^{p-1}}{p^p {\binom{2p-2}{p-1}} } \mathbbm{1}_{\{n=2p\}}
\]
is critical, and it is in fact regular critical because it has finite support. Since the weight of a map here only depends on its number of faces (or vertices), it is immediate that conditioning the distribution $B_{\mathbf{q}}$ to the set of maps with $n$ face yields the uniform $2p$-angulation with $n$ faces. We thus obtain the following.
\begin{prop}[Uniform infinite $2p$-angulation] Let $p\geq 2$ and, for $n\in \N$, let $(M_n,E_n)$ be a uniform rooted map amongst the set of rooted $2p$-angulation with $n$ faces. Then $(M_n,E_n)$ converges locally in distribution as $n$ goes to infinity, the limit being a random rooted map which we call the \emph{uniform infinite $2p$-angulation}.
\end{prop}
In the case where $2p=4$, we obtain the local convergence in distribution of large uniform quadrangulation to the UIPQ which was first obtained by Krikun in \cite{Krikun}. In fact our method here ends up being essentially the same as that of \cite{CMM}, where we have used the BDFG bijection in a situation where the simpler Cori-Vauquelin-Schaeffer bijection would have sufficed.

\medskip

\noindent\textit{Uniform infinite $2p+1$-angulation} Let $p\in\N$ and consider $2p+1$-angulations. It follows from the relation $\#V(m)= (p-1/2)\#F(m)$ that a $2p+1$-angulation must have an even number of faces. As in the even case, a uniform $2p+1$-angulation can be seen as a conditioned Boltzmann-distributed random map for the weight sequence $\mathbf{q}$ defined by 
	\[q_n=\alpha \mathbbm{1}_{\{n=2p+1\}},
\]
for any positive number $\alpha$. It has been shown in \cite{CLGM}, Proposition A.2 that there is one value of $\alpha$ which makes this sequence regular critical. Theorem \ref{chap4:thcvcartes} then gives us the following.

\begin{prop}[Uniform infinite $2p+1$-angulation] Let $p\in \N$ and, for $n\in \N$, let $(M_n,E_n)$ be a uniform rooted map amongst the set of rooted $2p+1$-angulation with $2n$ faces. Then $(M_n,E_n)$ converges locally in distribution as $n$ goes to infinity, to a random rooted map called the \emph{uniform infinite $(2p+1)$-angulation}.
\end{prop}

\subsection{Uniform planar maps}

It has been shown in \cite{MN14} that a uniform map chosen amongst the set of rooted maps with $n$ edges converges locally in distribution. Our methods also allow us to get this result, and that it is also true if we take a uniform pointed and rooted map.
\begin{prop}\label{prop:cvuipm} \begin{itemize}
\item[(i)] For $n\in\N$, let $(M_n,E_n,R_n)$ be a uniform random variable in the set of rooted and pointed maps with $n$ edges. Then $(M_n^{\bul},E_n^{\bul})$ converges locally in distribution to an infinite map called the \emph{uniform infinite planar map} (UIPM).
\item[(ii)] For $n\in\N$, let $(M_n,E_n)$ be a uniform random variable in the set of rooted maps with $n$ edges. Then $(M_n,E_n)$ converges locally in distribution to the UIPM.
\end{itemize}
\end{prop}

\begin{proof} Let $\lambda=\frac{1}{2\sqrt{3}}$ and define the weight sequence $\mathbf{q}$ by $q_n=\lambda^n$ for $n\in\N$. Given any map $(m,e,r)$, notice that $W_{\mathbf{q}}=\lambda^{\sum_{f\in \mathcal{F}_m} \deg(f)}=\lambda^{\#E(m)}$, where $\#E(m)$ is the number of edges of $m$. Thus, assuming that $\mathbf{q}$ is admissible (which we prove in the following), conditioning a map with distribution $B_{\mathbf{q}}$ (resp. $B_{\mathbf{q}}^{\emptyset}$) on having $n$ edges gives a uniform rooted map (resp. uniform rooted and pointed map) with $n$ edges.

Since we clearly have $d_F=1$, all we need to do now is to prove that $\mathbf{q}$ is regular critical. We start by computing the two generating functions $f^{\bul}$ and $f^{\dia}$. Recall the formulas
	\[\sum_{k=0}^{\infty} {\binom{k+p}{k}}x^k=\frac{1}{(1-x)^{p+1}}
\]
for $p\in\Z_+$ and $|x|<1$, as well as
	\[\sum_{k=0}^{\infty} {\binom{2k}{k}} z^k = \frac{1}{\sqrt{1-4z}}
\]
and
	\[\sum_{k=0}^{\infty} {\binom{2k+1}{k}} z^k = \frac{1-\sqrt{1-4z}}{2z\sqrt{1-4z}}
\]
for $|z|<1/4$. For $x\geq0$ and $y\geq0$, we have, with $z=\frac{\lambda^2x}{(1-\lambda y)^2}$:

\begin{align*}
f^{\bullet}(x,y)&=\sum_{k,k'}\frac{(2k+k'+1)!}{k!(k+1)!(k')!}\lambda^{2+2k+k'}x^k y^{k'} \\
                &=\lambda^2\sum_{k,k'}{\binom{2k+1}{k}}{\binom{2k+k'+1}{k'}}(\lambda^2x)^k(\lambda y)^{k'} \\
                &=\lambda^2\sum_k {\binom{2k+1}{k}} \frac{(\lambda^2x)^k}{(1-\lambda y)^{2k+2}} \\
                &=\lambda^2\frac{1}{(1-\lambda y)^2}\sum_k {\binom{2k+1}{k}} z^k \\
                &=\lambda^2\frac{1}{(1-\lambda y)^2} \frac{1-\sqrt{1-4z}}{2z\sqrt{1-4z}}.
\end{align*}

and
\begin{align*}
f^{\dia}(x,y)&=\lambda\sum_{k,k'}\frac{(2k+k')!}{(k!)^2!(k')!}(\lambda^2x)^k(\lambda y)^{k'} \\
             &=\lambda\sum_{k,k'}{\binom{2k}{k}}{\binom{2k+k'}{k'}}(\lambda^2x)^k(\lambda y)^{k'} \\
             &=\lambda\sum_k {\binom{2k}{k}} \frac{(\lambda^2x)^k}{(1-\lambda y)^{2k+1}} \\
             &=\lambda\frac{1}{1-\lambda y} \frac{1}{\sqrt{1-4z}}.
\end{align*}

Both series then converge if and only if $4z<1$, otherwise said $\lambda^2x<4(1-\lambda y)^2$. We then let the reader check that $x=4/3$ and $y=1/\sqrt{3}$ satisfy the wanted conditions for criticality and, since they are not on the edge of the domain, we even have regular criticality.
\end{proof}

\subsection{Proof of Theorem \ref{chap4:thcvcartes}}
The proof of Theorem \ref{chap4:thcvcartes} starts with the proof of the case of rooted and pointed maps. This involves showing the convergence for maps conditioned to be null or positive by using the BDFG bijection and identifying the limiting map as the image of an infinite tree by the bijection, and then removing the conditionings. To go from pointed to non-pointed maps, we will follow ideas from \cite{BJM14} and \cite{Abraham} to show that if a map is conditioned to have $n$ faces or edges, then its number of vertices is well concentrated around a deterministic multiple of $n$, and thus biasing by it will not change the convergence.
\subsubsection{On the trees associated to $B_{\mathbf{q}}$}\label{sec:mapperiod}
We want to investigate the periodic structure of Galton-Watson trees with ordered offspring distribution $\nu$. We thus let $\gamma_V=(1,0,0,0)$, $\gamma_E=(1,0,1,1)$ and $\gamma_F=(0,0,1,1)$ and, for $I\in\{V,E,F\}$, we let $d^{I}$ and $(\alpha^I_i)_{i\in[4]}$ be the periodicity factors given by Proposition \ref{per}. Note that we do not yet know that $d^I=d_I$, where $d_I$ was defined at the beginning of Section \ref{sec:cvcartes}. This is the main content of the following proposition:


\begin{lemma}\label{mapperiod}
We have, for $I\in\{V,E,F\}$:
\begin{itemize}
\item $d^I=d_I$.
\item $\alpha^I_1=\gamma^I_1$.
\end{itemize}
Moreover, if the weight sequence $\mathbf{q}$ is not only supported on $2\N$, then we also have
\begin{itemize}
\item $2\alpha^I_2\equiv \alpha^I_1\pmod d$.
\end{itemize}
\end{lemma}

\begin{proof} We will concentrate on the case where $I=V$, the other two cases having similar proofs. We first treat the bipartite case separately. In this case, types $1$ and $3$ alternate in tree, and it is straightforward that $d^V=gcd\big(\{m\in\N,\,q_{2m+2}>0\}\big)$ and that $\alpha_3=0$. We now assume not to be in this case.

It is immediate that $\alpha_3=0$ and $\alpha_4=\alpha_2$ because a vertex of type $1$ or $2$ can give birth to a single vertex of type $2$ or $4$, respectively.

Next, take $m\in\N$ such that $q_{2m+2}>0$. Using the fact that a vertex of type $3$ can give birth to $m$ vertices of type $1$, one obtains $m\equiv 0 \pmod {d^V}$. Thus $d^V$ divides the gcd of $\{m\in\N:\,q_{2m+2}>0\}$.

Now take an odd integer $m=2n+1$ such that $q_{m+2}=q_{2n+3}>0$. A vertex of type $3$ can then give birth to $n$ vertices of type $1$ and one vertex of type $2$, and a vertex of type $4$ can give birth to $n+1$ vertices of type $1$. We thus obtain $n+\alpha_2\equiv 0 \pmod d$ and $n+1 \equiv \alpha_2 \pmod d$. Combining these gives us $m=2n+1\equiv 0 \pmod {d^V}$ and $2\alpha_2 \equiv m+1 \equiv 1 \pmod {d^V}$.

We have thus shown $2\alpha_2 \equiv 1 \pmod {d^V}$ and that $d^V$ divides $d_V$. To show that they are equal we require some more refined analysis. 

Notice that for words $\mathbf{w}=(k,k',0,0)$ such that $\mu^{(3)}(\mathbf{w})>0$, we have $2k+k'\equiv 0 \pmod {d_V}$. Indeed, if $k'$ is even, letting $n=k+\frac{k'}{2}$, we then have $q_{2n+2}>0$, implying that $d'$ divides $n$, while if $k'$ is odd, we let $n=2k+k'$, and then $q_{n+2}>0$ and therefore $d'$ divides $n$. Similarly, if $\mu^{(4)}(\mathbf{w})>0$, then $2k+k'\equiv 1 \pmod {d_V}$. Applying this repeatedly to a tree $(\mathbf{t},\mathbf{e})$ such that $\pr^{(1)}_{\zeta}(\mathbf{T}\vdash \mathbf{t})>0$ and such that all its leaves are of type $1$ or $2$, one obtains $2k+k'=0 \pmod {d_V}$ where $k$ and $k'$ are respectively the number of leaves of type $1$ and $2$ in $\mathbf{t}$. Taking $(\mathbf{t},\mathbf{e})$ which has only one generation of type $1$, and we do obtain that $d_V$ divides every member of the support of $\mu_{1,1}$, which is all we need.
\end{proof}
\subsubsection{Infinite mobiles and the BDFG bijection}
We call \emph{infinite mobile} any infinite $4$-type labelled tree $(\mathbf{t},\mathbf{e},\mathbf{l})$ which satisfies the conditions of Section \ref{sec:thebijection}, which has a unique infinite spine and such that the labels of vertices of type $1$ of the spine do not have a lower bound. We let $\overline{\mathbb{T}}_M$ be the set of all finite and infinite mobiles, and split it in $\overline{\mathbb{T}}_M=\overline{\mathbb{T}}_M^+\cup\overline{\mathbb{T}}_M^0$ as before.

The BDFG bijection $\Psi$ can be naturally extended to $\overline{\mathbb{T}}_M$. Let $(\mathbf{t},\mathbf{e},\mathbf{l})\in\overline{\mathbb{T}}_M$, we let $(u_n)_{n\in\N}$ be the sequence of the elements of the spine. This sequence splits the tree in two: the part which is on the left-hand side of the spine, and the part which is on the right-hand side. To be precise, we say that $v\in\mathbf{t}$ is on the left-hand side of the spine if there exists three integers $n,k$ and $l$ and a sequence of integers $x$ such that $v=u_nkx$, $u_{n+1}=u_nl$ and $k<l$, and $v$ is on the right-hand side if we have the same, but with $k>l$. 

This splitting allows us to define a contour process, but it has to be indexed by $\Z$: since every subtree branching out of the spine is finite, we can let $\big(v(n)\big)_{n\in\N}$ be the contour process of the left-hand side and $\big(v(-n)\big)_{n\in\N}$ be the other half. This determines a unique sequence $\big(v(n)\big)_{n\in\Z}$. Since we have assumed that the labels of the vertices of type $1$ do not have a lower bound, the notion of successor we used for finite trees is still valid, and in fact, unlike in the case of a finite tree, we do not need to add an extra vertex. As in the finite case, we connect every vertex of type $1$ or $2$ to its successor, erase all the original edges of the tree, erase vertices of types $2$, $3$ and $4$, merging the two edges adjacent to every vertex of type $2$. This leaves us with an infinite map (by construction, the arcs do not intersect, while the following Lemma \ref{cont} implies that it is locally finite). We give this map a root edge which is determined with the same rules as in the finite case, however it is not pointed. We call this rooted map $\Psi(\mathbf{t},\mathbf{e},\mathbf{l})$.

\begin{lemma}\label{cont} The extended BDFG function $\Psi$ is continuous on $\overline{\mathbb{T}}_M$.
\end{lemma}

\begin{proof} Let $(\mathbf{t},\mathbf{e},\mathbf{l})$ be an infinite mobile. We assume $\mathbf{e}(\emptyset)=1$, the other case can be treated the same way. For $n\in\N$, we need to find $p \in\N$ such that, for another mobile $(\mathbf{t}',\mathbf{e}',\mathbf{l}')$, if $(\mathbf{t}_{\leq p},\mathbf{e}_{\leq p},\mathbf{l}_{\leq p})=(\mathbf{t}_{\leq p},\mathbf{e}_{\leq p},\mathbf{l}_{\leq p})$ then $B_{m,e}(n)=B_{m',e'}(n)$, where $(m,e,r)=\Psi(\mathbf{t},\mathbf{e},\mathbf{l})$ and $(m',e',r')=\Psi(\mathbf{t}',\mathbf{e}',\mathbf{l}')$. Let $s\in\N$ be large enough such that all the arcs in $B_{m}(n)$ connect vertices of $\mathbf{t}_{\leq s}$, let $x=\underset{v\in \mathbf{t}_{\leq s}}\inf \mathbf{l}(v)$ and let $u$ be any type $1$ vertex of the spine such that $\mathbf{l}(u)<x-1$. Notice now that there are no arcs connecting $\mathbf{t}_{\leq s}$ and the subtree above $u$. Indeed, the successor of any vertex of $\mathbf{t}_{\leq s}$ will be encountered below $u$ while, if $v$ is above $u$, $\mathbf{l}(v)\geq x$ would imply that its successor is also above $u$, while $\mathbf{l}(v)\leq x-1$ would make it impossible for its successor to be in $B_{(t,l)}(s)$. Taking $p$ to be the height of $u$ then ends the proof.
\end{proof}

\begin{lemma}\label{lem:planar} For any infinite mobile $(\mathbf{t},\mathbf{e},\mathbf{l})$, the infinite map $\Psi(\mathbf{t},\mathbf{e},\mathbf{l})$ is planar, in the sense that it can be embedded in the plane in such a way that bounded subsets of the plane only encounter a finite number of edges.
\end{lemma}

\begin{proof} We first start by embedding the mobile in the plane in a convenient way. We draw its infinite spine as the subset $\{0\}\times\Z_+$, where the child of $(0,n)$ is $(0,n+1)$ for $n\in\Z_+$. Starting from this, we can then embed the tree in $\Z\times\Z_+$ such that the second coordinate of a vertex is always its graph distance to the root, and also such that the children of any vertex $u$ always form a set of the type $\{(n,m),(n+1,m),\ldots,(n+k_{u}(\mathbf{t})-1,m)\}$, with $n\in\Z$ and $m\in\N$ and their first coordinates are in the correct order. With such an embedding, it is also apparent that there exists a continuous function $f:\R\to\R$ with is decreasing on $(-\infty,0]$ and increasing on $[0,+\infty)$, which has limit $+\infty$ at both $-\infty$ and $+\infty$ such that $\mathbf{t}$ is strictly above the graph of $f$.

We point out an important fact of the bijection: let $u$ be any corner of type $1$ or $2$ and let $v$ be its successor. Then, for any corner $w$ of type $1$ or $2$ which is encountered between $u$ and $v$ in the contour process of the mobile, the successor of $w$, which we call $x$, is then encountered between $w$ and $v$, and the arc between $w$ and $x$ is then enclosed between $\mathbf{t}$ and the arc connecting $u$ and $v$. From this fact, we obtain that all the arcs which connect two points on the left-hand side of the tree can be embedded without any issues: first draw the arcs connecting the line of successors starting at the root, and enclose in each of them the other necessary arcs.

The arcs which originate from the right-hand side of the tree are a more complex issue, because some of them might start very high on the right-hand side, go around a large part of the tree and end up high on the left-hand side. To make sure that these are well separated, we introduce for $n\in\Z_+$ the ``strip"
	\[S_n=\Big\{(x,y)\in\R^2: \ f(x)-n+1\leq y < f(x)-n\Big\}
\]
We now explore the right-hand side of the tree in counter-clockwise order and, when we encounter the $n$-th corner of type $1$ or $2$, we join it to $S_n$. We point out that it is possible to do this in such a way that second coordinate along the path is nondecreasing. We then do the same thing for the corner's successor, and then join both halves by a path which stays in $S_n$.

The paths we have drawn this way still do not intersect because of the ``enclosure" property as before, and this embedding is indeed such that bounded subsets of $\R^2$ only encounter a finite number of edges. This is because we have split these edges in parts which are in $S_n$, of which there is only one for every $n$, and parts which originate from vertices of the tree and have nondecreasing second coordinate, of which there are a finite amount in bounded subsets because there is a finite amount of vertices of $\mathbf{t}$ with bounded second coordinate.
\end{proof}

\subsubsection{Behaviour of the labels on the spine of the infinite tree} 
Let $(\mathbf{T,E,L})$ be a $4$-tree with law $\widehat{\pr}^{(1),(0)}_{\z,\nu}$ or the tree obtained from merging both components of a forest with distribution $\widehat{\pr}^{(2,2),(\frac{1}{2},\frac{1}{2})}_{\z,\nu}$ at their roots. The aim of this section is to show that it is an infinite mobile, that is, that the labels on the spine do not have a lower bound. Let us first describe it quickly.

The root of $\T$ has either type $1$ and label $0$, or type $2$ and label $1/2$, in which case it has (exceptionally) two children of type $4$, one of them (uniformly selected) being on the spine. The vertices which are not on the spine have offspring distribution $\zeta$, which was defined in Proposition \ref{bij} as the uniform ordering of $\mu$, while vertices which are on the spine have offspring $\widehat{\zeta}$, defined by \eqref{defzetahat}. The distribution $\widehat{\zeta}$ is itself the uniform ordering of a distribution $\widehat{\mu}$ on $(\Z_+)^4$ which we defined by
	\[\widehat{\mu}^{(i)}(k_1,k_2,k_3,k_4)=\frac{k_1b_1+k_2b_2+k_3b_3+k_4b_4}{b_i}\mu^{(i)}(k_1,k_2,k_3,k_4)
\]
for $i\in[4]$ and $k_1,k_2,k_3,k_4\in\Z_+$ and where $b_1,b_2,b_3,b_4$ are some positive numbers which depend on $\mathbf{q}$. The label displacement distribution $\nu^{(i)}_{\mathbf{w}}$ for a type $i\in[K]$ and a word $\mathbf{w}$ is then the uniform distribution on the set $D^{(i)}_{\mathbf{w}}$ which was defined in Proposition \ref{bij}.

\begin{lemma}\label{reversion} Let $i\in \{1,2,3,4\}$ and $\mathbf{w}\in\W_4$ such that $\zeta^{(i)}_{\mathbf{w}}>0$. Define the reversed word $\overset{\leftarrow}{\mathbf{w}}=(w_{|\mathbf{w}|},\ldots,w_1)$, and, for a label sequence $\mathbf{y}=(y_i)_{i\in[|\mathbf{w}|]}$, let $\overset{\leftarrow}{\mathbf{y}}=(-y_{|\mathbf{w}|},-y_{|\mathbf{w}|-1},\ldots,-y_1)$. The function which maps $\mathbf{y}$ to $\overset{\leftarrow}{\mathbf{y}}$ is a bijection between $D^{(i)}_{\mathbf{w}}$ and $D^{(i)}_{\overset{\leftarrow}{\mathbf{w}}}$, sets which are defined in Proposition \ref{bij}.
\end{lemma}

As a corollary, we get that, if $\mathbf{W}$ has distribution $\widehat{\zeta}^{(i)}$ for some $i$ and $\mathbf{Y}$ has distribution $\nu^{(i)}_{\mathbf{W}}$ conditionally on $\mathbf{W}$, then the pair $(\overset{\leftarrow}{\mathbf{W}},\overset{\leftarrow}{\mathbf{Y}})$ has the same distribution as $(\mathbf{W},\mathbf{Y})$.

\begin{proof} If $i=1$ or $i=2$ then the result is immediate, since $\overset{\leftarrow}{\mathbf{w}}=\mathbf{w}$ and $D^{(i)}_{\mathbf{w}}$ only has one element.

If $i=3$ or $i=4$, bijectivity of the map comes from the fact that reversing a sequence (and eventually changing the signs of its elements) is an involutive operation, and thus we only need to check that $\overset{\leftarrow}{\mathbf{y}}\in D^{(i)}_{\overset{\leftarrow}{\mathbf{w}}}$ for any displacement list $\mathbf{y}$, which is straightforward given the definitions, since $(-y_{|\mathbf{w}|+1-(i+1)})-(-y_{|\mathbf{w}|+1-i})=y_{|\mathbf{w}|-i+1}-y_{|\mathbf{w}|-i}$ for $i\in\{0,\ldots,|\mathbf{w}|\}$.
\end{proof}

\begin{lemma}\label{unbounded} Let, for $n\in \Z_+$, $U_n$ be the $(n+1)$-th vertex of type $1$ of the spine of $\mathbf{T}$. We then have
	\[\underset{n\in\N}\inf \ \mathbf{L}(U_n) = -\infty.
\]
\end{lemma}

\begin{proof}

Note that $U_n$ is well-defined for all $n\in\Z_+$, because the number of vertices of type $1$ on the spine of $\mathbf{T}$ is a.s. infinite. Indeed, if it were not the case then all the vertices on the spine after a certain height would have type $2$ or $4$, but since a vertex of type $4$ has positive probability of having at least one child of type $1$, having an infinite sequence of vertices $2$ and $4$ has probability $0$.

Notice then that $\big(\mathbf{L}(U_n)\big)_{n\in\Z_+}$ is in fact a centered random walk in $\mathbf{Z}$. It is a random walk because of the construction - the set of descendants of a vertex of type $1$ of the spine with label $k$ will have distribution $\widehat{\pr}^{(1,k)}_{\z,\nu}$. We can see that it is centered thanks to Lemma \ref{reversion}. Define the mirrored tree $(\overset{\leftarrow}{\mathbf{T}},\overset{\leftarrow}{\mathbf{E}},\overset{\leftarrow}{\mathbf{L}})$ by reversing the order of all the offspring of $\mathbf{T}$. To precise, if $u=u_1u_2\ldots u_n\in\mathbf{T}$, then let, for $i\in[n]$, $v_i=k_{u_1\ldots u_{i-1}}-i+1$ and let then $\overset{\leftarrow}{u}=v_1\ldots v_n$. Let then $\overset{\leftarrow}{\mathbf{E}}(\overset{\leftarrow}{u})=\mathbf{E}(u)$ and, define the labels $\overset{\leftarrow}{\mathbf{L}}$ on $\overset{\leftarrow}{\mathbf{T}}$ by $\overset{\leftarrow}{\mathbf{L}}(\emptyset)=\mathbf{L}(\emptyset)$ and, for all $u$, $\mathbf{y}_{\overset{\leftarrow}{u}}=\overset{\leftarrow}{\mathbf{y}_u}$ (as defined in Lemma \ref{reversion}). Since, for $i\in[4]$ and $\mathbf{w}\in\W_4$ the distribution $\zeta^{(i)}$ is the uniform ordering of $\mu^{(i)}$ and $\nu^{(i)}_{\mathbf{w}}$ is uniform on $D^{(i)}_{\mathbf{w}}$, we obtain from Lemma \ref{reversion} that $(\overset{\leftarrow}{\mathbf{T}},\overset{\leftarrow}{\mathbf{E}},\overset{\leftarrow}{\mathbf{L}})$ has the same distribution as $(\mathbf{T},\mathbf{E},\mathbf{L})$. In particular, $\mathbf{L}(U_1)-\mathbf{L}(U_0)$ has the same distribution as $\mathbf{L}(U_0)-\mathbf{L}(U_1)$, making its distribution centered. In particular, the centered random walk $\big(\mathbf{L}(U_n)\big)_{n\in\Z_+}$ then has a.s. no upper or lower bounds, for example by \cite{Kallenberg}, Theorem 8.2.


\end{proof}

\subsubsection{Removing conditionings}
Take once again $I\in\{V,E,F\}$. We need for this section some extra notation: for $n\in\N$, $\mathcal{M}_n$ is the set of pointed and rooted maps $(m,e,r)$ with $\#I(m)=n$, $\mathcal{M}_n^+$ and $\mathcal{M}_n^0$ are the analogous sets of positive and null maps. The probability measures $B_{\mathbf{q}}^n$, $B_{\mathbf{q}}^{n,+}$ and $B_{\mathbf{q}}^{n,0}$ are also the associated conditioned versions of $B_{\mathbf{q}}.$

The work done in the previous sections shows that maps with distribution $B_{\mathbf{q}}^{n,+}$ and $B_{\mathbf{q}}^{n,0}$ converge in distribution along the subsequence $(\alpha_I+d_In)_{n\in\N}$ (considering $B_{\mathbf{q}}^{n,0}$ only in the non-bipartite case). To show that maps with distribution $B_{\mathbf{q}}^n$ converge, all that is left for us to do is to show that the two quantities $B_{\mathbf{q}}^n (\mathcal{M}_n^+)$ and $B_{\mathbf{q}}^n( \mathcal{M}_n^0)$ converge (along the same subsequence). Since $2B_{\mathbf{q}}^n (\mathcal{M}_n^+)+B_{\mathbf{q}}^n( \mathcal{M}_n^0)=1$, we can in fact restrict ourselves to showing that the quotient $\frac{B_{\mathbf{q}}^n (\mathcal{M}_n^+)}{B_{\mathbf{q}}^n( \mathcal{M}_n^0)}$ converges. Elementary calculations on conditionings give us

\begin{align*}
\frac{B_{\mathbf{q}}^n (\mathcal{M}_n^+)}{B_{\mathbf{q}}^n( \mathcal{M}_n^0)}   &=\frac{B_{\mathbf{q}}\Big((M,E,R)\in\mathcal{M}^+\;|\;(M,E,R)\in\mathcal{M}_n\Big)}{B_{\mathbf{q}}\Big((M,E,R)\in\mathcal{M}^0\;|\;(M,E,R)\in\mathcal{M}_n\Big)} \\
                                   &=\frac{B_{\mathbf{q}}^+\Big((M,E,R)\in\mathcal{M}_n\Big)}{B_{\mathbf{q}}^0\Big((M,E,R)\in\mathcal{M}_n\Big)}\frac{B_{\mathbf{q}}(\mathcal{M}^+)}{B_{\mathbf{q}}(\mathcal{M}^0)}.
\end{align*}

Recall that, in the BDFG bijection, the number of vertices of the map is exactly one more than the number of vertices of type $1$ in the mobile. As a consequence, we have
	\[\frac{B_{\mathbf{q}}^+\Big((M,E,R)\in\mathcal{M}_n\Big)}{B_{\mathbf{q}}^0\Big((M,E,R)\in\mathcal{M}_n\Big)}
	           =\frac{\pr^{(1)}_{\zeta} \big( |\mathbf{T}|_{\gamma_I}=n-1\big)}{\pr^{(2,2)}\big( |\mathbf{F}|_{\gamma_I}=n-1)}.
\]
We then deduce from \eqref{byword} and Lemma \ref{mapperiod} that this quotient indeed converges as $n$ converges to infinity, along the $(\alpha_I+dn)_{n\in\N}$ subsequence.

\subsubsection{The non-pointed case}

We follow here the method given in \cite{Abraham}, Section 6. We now assume that $\mathbf{q}$ is regular critical, and that $I\in\{E,F\}$, and $(M_n,E_n,R_n)$ is still a map with distribution $B_{\mathbf{q}}$ conditioned on $\#I(M_n)=n$. We show that the number of vertices of $M_n$ is concentrated around a multiple of $n$.

\begin{lemma}\label{lem:unpointing} There exists $m>0$ such that, for $\delta>0$, there exists $C_{\delta}>0$ such that, for $n\in\N$ large enough,
	\[\pr\Big[\big|\#V(M_n) - m n\big|>\delta n\Big]\leq \exp(-C_{\delta}n)
\]
\end{lemma}

This lemma is itself a consequence of this similar result on trees:

\begin{lemma}\label{lem:concentrationtrees} Let $K\in\N$ and let $\nu$ be a non-degenerate, irreducible and regular critical $K$-type ordered offspring distribution. Let $\gamma\in(\Z_+)^K$ and $\gamma'\in(\Z_+)^K$ be two size measuring vectors. There then exists $m>0$ such that, for any type $i\in[K]$ and $\delta>0$, there exists $c_{\delta}>0$ such that, for $n\in\N$ large enough,
	\[\pr^{(i)}\big(|\T|_{\gamma}=n,\Big||\T|_{\gamma'}-mn\Big|\geq \delta n\big)\leq \exp(-c_{\delta}n)
\]
\end{lemma}

\begin{proof} We treat the case where $\gamma_k=\mathbbm{1}_{k=i}$ ($i$ being the type of the root) and $\gamma'_k=\mathbbm{1}_{k=j}$, and leave the generalisation to all $\gamma$ and $\gamma'$ to the reader. In this case, the value of $m$ we are looking for is in fact $a_j/a_i$, the average of the measure $\xi_{j,i}$ defined in Section \ref{gen1}.

Consider an infinite sequence of i.i.d trees $(\mathbf{T}_n)_{n\in\N}$ with distribution $\pr^{(i)}$ and list the vertices of type $i$ of these trees in lexicographical order. For $n$ in $\N$, call $A_n$ the total number of vertices of type $j$ placed between the $n$-th element of this list and the next generation of type $i$. Notice now that the event where $|\#_j\T_1-\frac{a_j}{a_i} n|>\delta n $ and $\#_1\T_1=n$ is included in the event where $|A_1+\ldots+A_n - \frac{a_j}{a_i} n|>\delta n$, thus giving us 
	\[\pr^{(1)}\Big(|\#_i\mathbf{T}-\frac{a_j}{a_i} n|>\delta n \;\bigcap\; \#_1\mathbf{T}=n\Big)\leq \pr\big(|A_1+\ldots+A_n - \frac{a_j}{a_i} n|>\delta n\big),
\]
Since these variables are i.i.d with distribution $\xi_{j,i}$ and have an exponential moment thanks to Proposition \ref{moments}, point (iv), we can apply Cram\'er's theorem which gives us a constant $c_{\delta}>0$ such that, for large enough $n\in\N$,
	\[\pr\big(|A_1+\ldots+A_n - \frac{a_j}{a_1} n|>\delta n\big) \leq \exp(-c_{\delta}n),
\]
thus ending the proof.
\end{proof}

Proving Lemma \ref{lem:unpointing} from Lemma \ref{lem:concentrationtrees} then simply consists in applying the BDFG bijection and noticing that conditioning on events of the form $\{|\T|_{\gamma}=n\}$ does not change anything here, because the probability of such an event is of order $n^{-3/2}$, the inverse of which can be absorbed in the exponential factor. \qed


\bigskip

\noindent\textit{End of the proof of Theorem \ref{chap4:thcvcartes}, second part.} Since we will simultaneously manipulate pointed and non-pointed maps, we change our notation slightly here: $(M_n,E_n,R_n)$ is a rooted and pointed map with distribution $B_{\mathbf{q}}$ conditioned on $\#I(M_n)=n$ and $(M_n^{\emptyset}, E_n^\emptyset)$ is a rooted map with distribution $B_{\mathbf{q}}^\emptyset$ conditioned on $\#I(M_n^\emptyset)=n$. As mentioned earlier, the distribution of $(M_n,E_n)$ is that of a biased version of $(M_n^{\emptyset}, E_n^\emptyset)$. We write this inversely: for bounded functions $F$, we have

   \[\mathbb{E}\big[F(M_n^\emptyset,E_n^\emptyset)\big]=\mathbb{E}[\frac{1}{\#V(M_n)}]^{-1}\mathbb{E}\big[\frac{F(M_n,E_n)}{\#V(M_n)}\big].
\]
If we let $X_n=mn(\#V(M_n))^{-1}$, we get the statement
	\[\mathbb{E}\big[F(M_n^\emptyset,E_n^\emptyset)\big]=\mathbb{E}\big[F(M_n,E_n)\frac{X_n}{\mathbb{E}[X_n]}\big],
\]
which yields
	\[\mathbb{E}\big[|F(M_n^\emptyset,E_n^\emptyset)-F(M_n,E_n)|\big]=\mathbb{E}\Big[F(M_n,E_n)\Big|1-\frac{X_n}{\mathbb{E}[X_n]}\Big|\Big].
\]
Proving that $X_n$ converges to $1$ in $L^1$ will then end the proof. Take $\delta>0$ and write
	\[\mathbb{E}[|X_n-1|]\leq \delta + \mathbb{E}[|X_n-1|\mathbbm{1}_{\{|X_n-1|>\delta\}}].
\]
Let $\varepsilon=\delta(1-\delta)^{-1}$, such that the event $\{|X_n-1|>\delta\}$ is included in the event $\{|\#V(M_n) - m n|>\varepsilon n\}$. Recalling that $X_n\leq mn$, we then have
	\[\mathbb{E}[|X_n-1|]\leq \delta + (mn+1)\pr(|Y_n - m n|>\varepsilon n),
\]
and Lemma \ref{lem:unpointing} ends the proof of Proposition \ref{prop:cvuipm}.
\qed

\section{Recurrence of the infinite map}
The aim of this section if to prove the following:
\begin{theo}\label{recurrence} The random rooted graph $(M_{\infty},E_{\infty}^+)$ is almost surely recurrent.
\end{theo}

Our principal tool for the proof will be the main result of \cite{GGN}: since $(M_{\infty},E_{\infty}^+)$ is the limit in distribution of $\big((M_n,E_n^+),n\in\N\big)$, and since $E_n^+$ is chosen according to the stationary distribution on $M_n$ (that is, a vertex is chosen with probability proportional to its degree, i.e. its number of adjacent edges), then Theorem 1.1 of \cite{GGN} states that if we can find positive constants $\lambda$ and $C$ such that, for all $n\in\N$,
	\[\pr(\mathrm{deg}(E^+)\geq n) \leq C\e^{-\lambda n},
\]
then Theorem \ref{recurrence} will be proven.

We invite the reader to read Appendix \ref{sec:EI} where we discuss a few elementary results concerning random variables with such exponential tails.

\subsection{The case of positive maps}
Picture a mobile $(\mathbf{T,E,L})$ with distribution $\widehat{\pr}^{(1),(0)}_{\zeta,\nu}$: it has an infinite spine, and on its right and left sides are grafted some finite trees. Since the BDFG bijection makes $\emptyset$ into $e^+$, we will want to show that $\emptyset$ has an exponentially integrable number of successors and is the successor of an exponentially integrable number of vertices. We start with a simplified case.


\begin{prop}\label{unsousarbre} Let $\mathbf{A}$ be a mobile with distribution $\pr_{\z,\nu}^{(1),(0)}$ conditioned to the event where $\emptyset$ has exactly one child. Let $X$ be the number of corners of $\mathbf{A}$ for which $\emptyset$ is the successor. Then $X$ is an $EI(\lambda)$ variable for a certain $\lambda>0$.
\end{prop}

\begin{proof} Recall that $X$ is the number of corners labelled $1$ or $\frac{1}{2}$ met before encountering a vertex labelled $0$ while circling counterclockwise around the tree $\mathbf{A}$. We will separately treat corners of types $1$ and $2$.

Let $X_1$ be the number of corners of type $1$ encountered. We claim that, for all $n$,
\begin{equation}\label{SErec}
\pr(X_1=n \; | \; X_1\geq n)\geq \alpha(1-\frac{1}{Z^+}),
\end{equation}
where $\alpha>0$ is the probability that, given a vertex of type $3$ labelled $1$, its rightmost offspring is of type $1$ and has label $0$. The fact that $\alpha$ is strictly positive comes from the fact that there exists $i\geq 3$ such that $q_i>0$. In the case where such an $i$ is different from $3$, the vertex of type $3$ can have offspring with at least one child of type $1$, the uniform ordering of the offspring means that this child can be the rightmost one, and the distribution of the label displacements shows that it can have label $0$. For the case where $q_3>0$ and $q_i=0$ for $i\geq 4$, the type $3$ vertex can have a unique child of type $2$ with label $\frac{1}{2}$, which can have a unique child of type $4$ which can have a unique child of type $1$ with label $0$.

Inequality \eqref{SErec} is obtained by recalling from Proposition \ref{bij} that the offspring of vertices of type $1$ is only made of vertices of type $3$, and that their number follows a geometric distribution with parameter $1-\frac{1}{Z^+}$. Thus, whenever we visit a corner of a type $1$ vertex with label $1$, there is a $1-\frac{1}{Z^+}$ chance that this vertex has another child. This immediately gives us \eqref{SErec}, and a simple induction shows that $X_1$ is indeed an $EI$ variable.

\medskip

Let now $X_2$ be the number of \emph{vertices} of type $2$ with label $\frac{1}{2}$ encountered before the first vertex labelled. We insist that we count each vertex exactly once, when we meet them for the first time on the counter-clockwise exploration path. Then the same argument as for vertices of label $1$ shows that $\pr(X_2=n \; | \; X_2\geq n)\geq \alpha'$ for some strictly positive $\alpha'$, and $X_2$ is an $EI$ variable.

\medskip

Since $X\leq X_1 + 2X_2$, we now have our conclusion.
\end{proof}

The following lemma provides some additional on the structure $\mathbf{T}$.

\begin{lemma}\label{liengeo} Let $n\in\Z_+,$ and let $V$ be the $n$-th vertex of the spine of $\mathbf{T}$ to have type $1$. Let also $N_r$ and $N_l$ be the numbers of subtrees rooted at $v$ on the right and left sides of the spine. These variables are i.i.d. and their common distribution is geometric with parameter $1-\frac{1}{Z^+}$.
\end{lemma}
\begin{proof} By combining Proposition \ref{bij} and Proposition \ref{definftree}, we obtain that the total offspring $N$ of $V$ follows a size-biased geometric distribution: we have
	\[\pr(N=k)=\frac{k}{(Z^+)^2}(1-\frac{1}{Z^+})^{k-1}
\]
for $k\geq 1$. Recall also that the child of $V$ which is on the spine is chosen uniformly amongst the offspring of $V$. We thus have
	\[\pr(N_l=k,N_r=k')=\frac{\pr(N=1+k+k')}{1+k+k'}=(\frac{1}{Z^+})(1-\frac{1}{Z^+})^{k+k'},
\]
ending the proof.
\end{proof}

\noindent\textbf{Proof of Theorem \ref{recurrence} for positive maps:} First off, by Lemma \ref{liengeo}, we know that $\emptyset$ has an EI amount of children, since geometric variables are EI, and therefore has an EI amount of successors. Next, look at all the subtrees of $\mathbf{T}$ which are rooted at $\emptyset$, excluding the subtree containing the spine. These are in EI amount, all independent, and, by Proposition \ref{unsousarbre}, the root $\emptyset$ is connected to an EI amount of vertices in each of them. Lemma \ref{sommeEI} allows to combine all of this: outside of the subtree containing the spine, $\emptyset$ is connected to an EI amount of vertices. Thus we now only need to prove a variation of Proposition \ref{unsousarbre} for this very subtree. This is done in the same way since, when doing the counterclockwise exploration process, the number of children of a vertex of type $1$ on the spine is still geometric by Lemma \ref{liengeo}, while vertices of type $2$ only correspond to one corner. \qed


\subsection{The case of null maps}
The situation for null maps is slightly different, because the vertex $E^+$ is no longer the root of the mobile. Consider a mobile $(\mathbf{T,E,L})$ obtained by merging at their roots the two components of a forest with distribution $\widehat{\pr}^{(2,2),(\frac{1}{2},\frac{1}{2})}_{\z,\nu}$, and let $(M,E,R)$ be the map obtained after applying the BDFG bijection. Recall that $E^+$ is the first vertex of type $1$ and label $0$ encountered when running the clockwise countour process of $\mathbf{T}$. Note that it is either on the spine or on its left side.

\medskip
An adaptation of the reasoning used in the previous section will work and give us that the number of vertices $E^+$ is connected to is indeed EI. First, for the number such vertices which are descendants of $E^+$ in $\mathbf{T}$, we find ourselves exactly back to the positive case: if $E^+$ is not on the spine then we apply Proposition \ref{unsousarbre} to an EI number of subtrees rooted at $E^+$, and if $E^+$ is on the spine, we separate the subtrees on the left side of the spine, on the right side of the spine and the subtree containing the spine. Secondly, we look for points of which $E^+$ is the successor, but which are not descendants of $E^+$. These can be obtained by running both the clockwise and counter-clockwise contour processes, starting at the root, and stopping them the first time we reach a $0$ label. The same arguments as in the proof of Proposition \ref{unsousarbre} show that we encounter an EI number of vertices of labels $1$ and $\frac{1}{2}$ on the way, thus ending the complete proof of Theorem \ref{recurrence}.\qed

\appendix

\section{Around exponentially integrable variables}\label{sec:EI}
We recall here a few basic facts about non-negative random variables with exponential moments.
\medskip
Let $X$ be a nonnegative random variable. We say that $X$ is \emph{exponentially integrable with parameter $\lambda>0$} (which we shorten as $EI(\lambda)$, and simply $EI$ if we are not interested in the value of $\lambda$) if we have
	\[\mathbb{E}[\e^{\lambda X}] < \infty.
\]
The use of Markov's inequality shows that this implies that the tail of $X$ is bounded by an exponential with parameter $\lambda$:
	\[\forall n\in\N, \pr(X\geq n) \leq \mathbb{E}[\e^{\lambda X}]\e^{-\lambda n}.
\]
The converse is no quite true, but almost is: if the tail of $X$ is bounded by an exponential with parameter $\lambda$, then $X$ is $EI(\lambda')$ for $\lambda' < \lambda$.

\begin{lemma}\label{sommefinieEI} If $X$ and $Y$ are two $EI(\lambda)$ variables then $X+Y$ is $EI(\lambda')$ for $\lambda'<\frac{\lambda}{2}$.
\end{lemma}

\begin{proof}
Just bound $\pr(X+Y>n)$ by $\pr(X>\frac{n}{2})+\pr(Y>\frac{n}{2})$.
\end{proof}

With an extra independence assumption, one can also do sums with a random amount of terms:

\begin{lemma}\label{sommeEI} Let $(X_i)_{i\in\N}$ be i.i.d nonnegative variables which are $EI(\lambda)$ for some $\lambda>0$. Let $N$ be an independent integer-valued variable which is $EI(\mu)$ for some $\mu>0$. If $\mathbb{E}[\e^{\lambda X_1}]\leq \e^{\mu}$ (which is always possible by taking $\lambda$ small enough), the the variable
	\[Y=\sum_{i=1}^N X_i
\]
is also $EI(\lambda)$.
\end{lemma}

\begin{proof} Conditioning on $N$ and integrating with respect to all of the $X_i$, one immediately obtains
	\[\mathbb {E}[\e^{\lambda Y}]=\mathbb {E}\Big[\mathbb {E}[e^{\lambda X_1}]^N\Big],
\]
and this is enough.
\end{proof}
This could of course be generalized to the case where the $(X_i)$ do not have the same distribution, but uniformally bounded exponential moments - we will not need such a generalization.

\section*{Acknowledgments} This paper has been in the works for a very long time, and has evolved a lot in scope since its inception. The author would like to thank Gr\'egory Miermont for many discussions and countless proofreadings, Gregory Lawler for an exchange concerning ratio limit theorems for random walks, Nicolas Curien and Svante Janson for reading the preliminary version which appeared in his thesis, and finally the anonymous referee for a very careful proofreading.

\bibliographystyle{siam}
\bibliography{bib}
\end{document}